\documentclass{article}
\usepackage{amsmath}
\usepackage{amssymb}
\usepackage{fullpage}
\usepackage{algpseudocode}
\usepackage[ruled,vlined,linesnumbered]{algorithm2e}
\DontPrintSemicolon
\SetKwInput{Input}{Input}
\SetKwInput{Output}{Output}
\SetKw{To}{ to }
\SetKw{Print}{Print }
\SetAlFnt{\small}
\usepackage{graphics}
\usepackage{booktabs}
\usepackage{cite}
\usepackage{bbm}
\usepackage{tikz}
\usepackage{bm}
\usepackage{appendix}
\usepackage{hyperref}
\usepackage[margin=3cm]{geometry}

\usepackage{algorithm2e}

\usepackage{authblk}

\usepackage{graphicx}
\usepackage{subfig}

\usepackage{multirow}
\usepackage{rotating}
\usepackage{verbatim}
\usepackage{optidef}
\usepackage{cleveref}
\usepackage{array}
\usepackage{tabularx}
\usepackage{arydshln}
\usepackage{mathtools}
\usepackage{easybmat}
\usetikzlibrary{positioning}

\usepackage{enumerate}
\usepackage{enumitem}
\usepackage{rotating}
\usepackage{pdflscape}

\usepackage{tikz,forest} 
\usetikzlibrary{arrows.meta}
\forestset{
    .style={
        for tree={
            base=bottom,
            child anchor=north,
            align=center,
            s sep+=1cm,
    straight edge/.style={
        edge path={\noexpand\path[\forestoption{edge},thick,-{Latex}] 
        (!u.parent anchor) -- (.child anchor);}
    },
    if n children={0}
        {tier=word, draw, thick, rectangle}
        {draw, diamond, thick, aspect=2},
    if n=1{%
        edge path={\noexpand\path[\forestoption{edge},thick,-{Latex}] 
        (!u.parent anchor) -| (.child anchor) node[pos=.2, above] {Y};}
        }{
        edge path={\noexpand\path[\forestoption{edge},thick,-{Latex}] 
        (!u.parent anchor) -| (.child anchor) node[pos=.2, above] {N};}
        }
        }
    }
}

\DeclarePairedDelimiterX{\Set}[1]\{\}{\setargs{#1}}
\NewDocumentCommand{\setargs}{>{\SplitArgument{1}{|}}m}
{\setargsaux#1}
\NewDocumentCommand{\setargsaux}{mm}
{\IfNoValueTF{#2}{#1}{\nonscript\,#1\nonscript\;\delimsize\vert\allowbreak \nonscript\:\mathopen{}#2\nonscript\,}}

\newcommand*\xbar[1]{%
  \hbox{%
    \vbox{%
      \hrule height 0.5pt 
      \kern0.5ex
      \hbox{%
        \kern-0.1em
        \ensuremath{#1}%
        \kern-0.1em
      }%
    }%
  }%
} 

\usepackage{tikz,pgfplots,xpatch}
\usetikzlibrary{shapes.multipart,backgrounds,shapes,fit}
\usetikzlibrary{arrows,decorations.markings,shapes.arrows,arrows.meta}

\newcommand{\e}{\varepsilon}
\newcommand{\red}[1]{{\color{red} #1}}

\newcommand{\E}{\mathbb{E}}

\newcommand{\R}{\mathbb{R}}

\renewcommand{\P}{\mathbb{P}}

\DeclareMathOperator{\conv}{\operatorname{conv}}

\DeclareMathOperator{\rank}{\operatorname{rank}}

\newcommand{\ip}{\iprods}

\definecolor{mypink3}{cmyk}{0, 0.7808, 0.4429, 0.1412}

\newcommand{\blue}[1]{{\color{blue} #1}}

\newcommand{\cN}{\mathcal{N}}

\DeclareMathOperator{\argmin}{\operatorname{argmin}}
\DeclareMathOperator{\argmax}{\operatorname{argmax}}
\newcommand{\norm}[1]{\left\lVert#1\right\rVert}

\newcommand{\av}[0]{\mathbf{a}}
\newcommand{\cv}[0]{\mathbf{c}}

\newcommand{\vv}[0]{\mathbf{v}}
\newcommand{\xv}[0]{\mathbf{x}}
\newcommand{\yv}[0]{\mathbf{y}}

\newcommand{\bv}[0]{\mathbf{b}}

\newcommand{\zv}[0]{\mathbf{z}}

\newcommand{\ones}{{\bm{1}}}
\newcommand{\LP}{\text{LP}}

\newcommand{\vol}{\textup{vol}}
\newcommand{\cM}{\mathcal{M}}

\newcommand{\normm}[1]{\big\lVert#1\big\rVert}
\newcommand{\iprods}[2]{\langle {#1}, {#2} \rangle}

\newcommand{\colij}{\av^{(i)}_j}

\newtheorem{thm}{Theorem}[section]

\newtheorem{lemma}{Lemma}[section]

\usepackage{booktabs}

\newtheorem{cor}[thm]{Corollary}
\newtheorem{lem}[thm]{Lemma}

\newtheorem{claim}[thm]{Claim}

\newcommand{\nrc}{R}

\newenvironment{proof}[1][]
   {\textsc{Proof#1: }}
   {\hfill$\Box$\smallskip}

\newcommand{\iprod}[2]{\left\langle {#1}, {#2} \right\rangle}

\newcommand{\OPT}{\textup{OPT}}
\newcommand{\LAG}{\textup{DUAL}}

\newcounter{algorithmctr}[section]
\renewcommand{\thealgorithmctr}{\thesection.\arabic{algorithmctr}}
    {\refstepcounter{algorithmctr}\begin{list}{}{%
        \setlength{\rightmargin}{.05\linewidth}%
        \setlength{\leftmargin}{.05\linewidth}}%
        \item[]{\setlength{\parskip}{0ex}\bigskip\par%
         \nopagebreak%
         \underline{{\bf Algorithm \thealgorithmctr.} \emph{#1.}}}}%
    {{\setlength{\parskip}{-1ex}\nopagebreak\smallskip\par} \end{list}}

\makeatletter
\newsavebox\myboxA
\newsavebox\myboxB
\newlength\mylenA

\makeatother

\title{Probabilistic analysis of dual decomposition on two-stage stochastic integer programs}

\author[1]{Santanu S. Dey\thanks{Email: sdey30@gatech.edu}}
\author[2,3]{Marco Molinaro\thanks{Email: mmolinaro@microsoft.com}}
\author[1]{Jingye Xu\thanks{Email: jxu673@gatech.edu}}

\affil[1]{Georgia Institute of Technology, Atlanta, GA, USA}
\affil[2]{Microsoft Research, Redmond, WA, USA}
\affil[3]{PUC-Rio, Rio de Janeiro, Brazil}

\begin{document}

\maketitle

\begin{abstract}
Two-stage stochastic integer programs provide a powerful framework for modeling
decision-making under uncertainty, but they are notoriously difficult to solve at
scale due to their high dimensionality and intrinsic nonconvexity.
Decomposition-based algorithms such as Benders methods and Branch-and-Price (related dual
decomposition methods) have become standard computational approaches for such
problems and demonstrate excellent empirical performance in practice.
Despite their widespread use, however, existing theoretical guarantees are almost
exclusively based on worst-case analyses, which predict exponential convergence
behavior in the problem dimension and fail to explain the strong performance
observed in practice. In this paper, we present the first average-case analysis of Branch-and-Price
for a broad class of two-stage stochastic binary integer programs.
We study a stochastic-input model in which objective coefficients and constraint
matrices are drawn at random and right-hand-side vectors scale with the decision
dimension, while the number of constraints per scenario is fixed.
Under this model, we prove that, with high probability, Branch-and-Price explores
at most \( n^{O(\log s)} \) nodes, yielding a quasi-polynomial bound on the size of
the search tree in typical instances, where \( n \) denotes the decision dimension
and \( s \) the number of scenarios. A key ingredient of our analysis is an average-case bound on the integrality gap of
the natural linear programming (LP) relaxation. We show that this gap shrinks at rate
\( O\!\left(\frac{\log s \log^2 n}{n}\right) \) with high probability. This result is of
independent interest, as it implies that the integrality gap grows only logarithmically
with the number of scenarios on average. This helps explain why LP-based
heuristics work well in practice, even in large-scale stochastic settings with an enormous number of
scenarios, where LP relaxations with a large number of constraints are usually believed to be weak.
\end{abstract}

\section{Introduction}

Consider the following problem:
\begin{subequations}
\label{eq:IP}
\begin{align}
\OPT(\bv) := \max_{\xv} \ & s \cdot \iprod{\cv^{(0)}}{\xv^{(0)}} + \sum_{i \in [s]} \iprod{\cv^{(i)}}{\xv^{(i)}} \label{eq:obj1} \\
\text{s.t.}\ & A^{(0)} \xv^{(0)} + A^{(i)}\xv^{(i)} \leq \bv^{(i)}, \quad \forall i \in [s], \label{eq:sub}\\
& \xv^{(0)}, \xv^{(i)} \in \{0,1\}^{n}, \quad \forall i \in [s]. \label{eq:origLast}
\end{align}
\end{subequations}
where $A^{(i)} \in \mathbb{Q}^{m \times n}$ and $\cv^{(i)} \in \mathbb{Q}^{n}$ for all
$i \in [s] \cup \{0\}$.

Problem~\eqref{eq:IP} is a two-stage stochastic integer program, a widely used
framework for modeling decision-making under uncertainty \cite{kuccukyavuz2017introduction,liu2016decomposition,romeijndersastochastic,sahinidis2004optimization}.
The vector $\xv^{(0)}$ represents first-stage (here-and-now) decisions that must be
fixed before the uncertainty is realized, while $\xv^{(i)}$ denotes second-stage
(recourse) decisions associated with scenario $i \in [s]$.
Each scenario corresponds to a possible realization of the underlying uncertainty
and is characterized by its right-hand side vector $\bv^{(i)}$ and
scenario-dependent constraint matrix $A^{(i)}$.
In this paper, we focus on the setting in which all decision variables are binary
and the first-stage and second-stage variables have the same dimension.
These assumptions are made solely for notational simplicity and clarity of
exposition.
All results presented here extend in a straightforward manner to settings with
heterogeneous variable dimensions, nonuniform scenario weights, and more general
mixed-integer decision domains.

In practice, uncertainty is typically approximated by a finite set of scenarios
obtained via sampling, scenario reduction, or discretization.
Under this approximation, problem~\eqref{eq:IP} can be viewed as a deterministic
approximation of the original stochastic program.
Although the dimension of~\eqref{eq:IP} grows only linearly with the number of
scenarios~$s$, the associated computational burden often increases superlinearly.
In the worst case, memory requirements, solution time, and solver performance
deteriorate rapidly, and even evaluating a single candidate first-stage solution
may require solving a large number of scenario-specific integer subproblems.

Decomposition methods are therefore particularly appealing for solving
problem~\eqref{eq:IP}, as they exploit the inherent separability across scenarios \cite{boland2018combining,kim2022scalable,kuccukyavuz2017introduction,cifuentes2025lagrangian,rahmaniani2017benders,dey2018analysis}.
By decomposing the problem into smaller, more tractable subproblems, these methods
reduce memory usage, improve computational efficiency, and enable scalable solution
approaches for large-scale instances.
Two general decomposition paradigms are most commonly used in the literature for
two-stage stochastic integer programs: \emph{Benders decomposition} and
\emph{dual decomposition}. Benders decomposition is a cutting-plane approach that separates inequalities through
value-function approximations of the second-stage problem, whereas dual
decomposition reformulates the problem by introducing nonanticipativity
constraints and relaxing these constraints via Lagrangian multipliers.
For comprehensive surveys, we refer the reader to~\cite{rahmaniani2017benders} for Benders decomposition
and~\cite{bragin2024survey} for dual decomposition.

Despite their strong empirical performance, there remains a substantial gap
between the practical success of decomposition methods and the current theoretical
understanding of their convergence behavior for two-stage stochastic integer
programs.
Existing convergence guarantees are almost exclusively derived from worst-case
analyses, which yield exponential complexity bounds of order $\Theta(2^n)$ \cite{sun2024decomposition,zou2019stochastic}.
Such bounds typically rely on the observation that there are at most $2^n$ possible
first-stage integer solutions and that each iteration of a decomposition algorithm
eliminates at least one suboptimal candidate.
However, this pessimistic worst-case behavior is rarely observed in practice.
Extensive computational studies indicate that decomposition methods often converge
in only a small number of iterations, and in many cases even a single iteration
suffices to identify a high-quality or optimal solution \cite{kim2022scalable}.
This pronounced discrepancy motivates the need for a refined, average-case
analysis of the convergence behavior of decomposition methods.

In this paper, we focus on the dual decomposition approach.
Because problem~\eqref{eq:IP} is  nonconvex, there is generally a positive
duality gap between~\eqref{eq:IP} and the Lagrangian dual associated with its
nonanticipativity reformulation. 
As a result, additional branching is required in practice to close this gap and
recover an integral optimal solution.
This leads to algorithmic frameworks commonly referred to as
\emph{Branch-and-Price}, \emph{column generation}, or
\emph{Dantzig--Wolfe decomposition}~\cite{lubbecke2010column,desaulniers2006column}. In this paper, we study the behavior of these algorithms.

To move beyond worst-case analysis, we adopt the common perspective of considering random input instances and focus on average-case performance.
Specifically, we consider the following \textbf{stochastic-input model}:
\begin{enumerate}
    \item All entries of $\cv^{(i)}$ and $A^{(i)}$ are drawn independently and
    uniformly from $[0,1]$, for all $i \in [s] \cup \{0\}$.
    \item The right-hand-side vectors are given by
    $\bv^{(i)}_r = \beta^{(i)}_r \cdot 2n$, where
    $\beta^{(i)}_r \in (\tfrac{1}{4}, \tfrac{1}{2})$ for all
    $i \in [s]$ and $r \in [m]$.
\end{enumerate}

This random model has been considered in prior work (see~\cite{dyer1989probabilistic,dey2021branch}) and is widely
used to evaluate the performance of decomposition methods~\cite{angulo2016improving,bansal2024computational}.
Regarding the choice of the right-hand side, note that if
$\beta^{(i)}_r > \tfrac{1}{2}$ for all $i$ and $r$, then as $n \to \infty$ with $m$
fixed, all binary solutions of~\eqref{eq:IP} become feasible with high probability.
In this regime, the problem admits the trivial optimal solution in which all
decision variables equal one and thus becomes asymptotically uninformative.
For technical reasons (see Section~\ref{sec:bound_dual}), we additionally impose the lower
bound $\beta^{(i)}_r \ge \tfrac{1}{4}$, which is slightly stronger than assumptions
used in~\cite{dyer1989probabilistic}.
Instances with very small values of $\beta^{(i)}_r$ are typically less interesting,
as they admit only a limited number of feasible first-stage solutions.
It is a folklore observation that the most challenging regime occurs when
$\beta^{(i)}_r$ is close to $\tfrac{1}{2}$.
We therefore believe that the proposed stochastic-input model is representative
and well suited for evaluating the typical performance of decomposition-based
algorithms.

In this work, we assume access to an oracle that can exactly solve the
dual decomposition problem, and we measure the convergence behavior of
Branch-and-Price by the number of nodes explored in the Branch-and-Price tree
(Section~\ref{sec:BB} provides precise definitions). We emphasize that this
oracle only requires an optimization oracle for each individual scenario
subproblem of dimension $n$.
Following prior work~\cite{dyer1989probabilistic,borst2023integrality,dey2021branch},
we adopt a fixed-parameter regime in which the number of constraints in each
scenario subproblem is fixed at $m$, while the decision dimension $n$ and the
number of scenarios $s$ grow. 
Unlike previous studies, we do not fix the total number of constraints in
problem~\eqref{eq:IP}; instead, the total number of constraints grows linearly
with the number of scenarios.
Our main result can be stated informally as follows.

\begin{thm}\label{thm:BB_inf}
In the stochastic-input model, for sufficiently large $n$, with probability at
least $0.9999 - \frac{1}{n}$, the Branch-and-Price algorithm
(see Section~\ref{sec:BB}) solves~\eqref{eq:IP} after exploring at
most $n^{O_m(\log s)}$ nodes.
\end{thm}

Theorem~\ref{thm:BB_inf} follows from the combination of two results established
in the subsequent sections. The first result relates the size of the
Branch-and-Price tree to the integrality gap of~\eqref{eq:IP}. The second result
shows that this integrality gap is small with high probability under the
stochastic-input model. Taken together, these results yield the stated bound on
the number of nodes explored by the Branch-and-Price algorithm.

\begin{lem}\label{lem:inform_gap_to_BB}
Consider the stochastic-input model, and suppose that, with probability at least $1-\delta$, the
integrality gap of problem~\eqref{eq:IP} satisfies
$O_m\!\left( \frac{s \log s \log^2 n}{n} \right).$\footnote{We use the notation $O_m(g(n,s))$ to denote a function $f(n,s,m)$ such that for every fixed $m \ge 0$, there exists a constant $c$ such that $f(n,s,m) \le c \cdot g(n,s)$ for all sufficiently large $n,s$.}
Then, for sufficiently large $n$, with probability at least
$1 - \delta - \frac{1}{n}$, the Branch-and-Price algorithm solves~\eqref{eq:IP} after exploring at
most $n^{O_m(\log s)}$ nodes.
\end{lem}

\begin{thm}\label{thm:inform_int_gap}
In the stochastic-input model, for sufficiently large $n$, with probability at
least $0.9999$, the integrality gap of~\eqref{eq:IP} satisfies
$O_m\!\left( \frac{s \log s \log^2 n}{n} \right).$
\end{thm}

To the best of our knowledge, this work provides the first theoretical
average-case analysis of the Branch-and-Price algorithm, which has been a
standard approach for solving large-scale integer programs since the
1990s~\cite{lubbecke2010column,desaulniers2006column}. Our result helps close the gap between the widespread practical success of Branch-and-Price
and the lack of rigorous guarantees for its typical-case behavior, especially
in high-dimensional and multi-scenario settings.
Moreover, Theorem~\ref{thm:inform_int_gap} extends the classical average-case integrality gap
analysis of~\cite{dyer1989probabilistic} from single-block integer programs to
a much richer multi-block two-stage setting, while requiring only logarithmic
growth in the number of scenarios~$s$ (after normalizing by $\frac{1}{s}$).
This is of independent interest, as it shows that the LP relaxation remains
tight even when the number of scenarios grows, a regime typically regarded as
computationally prohibitive. In particular, it provides theoretical support for
LP-relaxation-based decomposition methods~\cite{angulo2016improving,chen2022generating}
by showing that their underlying relaxations continue to yield high-quality
dual bounds in large-scale two-stage problems with many scenarios.

The remainder of the paper is devoted to proving Theorem~\ref{thm:BB_inf} and is organized as follows.
Section~\ref{sec:BB} presents a detailed description of the Branch-and-Price algorithm. Sections~\ref{sec:number_node} and~\ref{sec:int_gap} prove Lemma~\ref{lem:inform_gap_to_BB} and Theorem~\ref{thm:inform_int_gap}, respectively.
\section{Dual Decomposition and Branch-and-Price Algorithm}
\label{sec:BB}


In this section, we provide a detailed description of a standard Branch-and-Price algorithm \cite{barnhart1998branch,lubbecke2010column,desaulniers2006column}. 
Similar to the standard Branch-and-Bound algorithm, which relies on LP relaxations, Branch-and-Price can be viewed as an enumeration procedure guided by a convex relaxation, where each node solves a dual decomposition problem (also known as column generation, and as ``Lagrangian relaxation'' in the stochastic programming literature). 

Since in later sections we will also consider Lagrangian relaxations of linear programs, we adopt the term \emph{dual decomposition relaxation} here to clearly distinguish between these two types of relaxations.

\paragraph{Dual decomposition relaxation.} This continuous relaxation is obtained by convexifying each combination of the first-stage block and a single scenario block separately (whereas the original problem is equivalent to jointly convexifying the first-stage block together with all scenario blocks). More precisely, it is formally defined as
\begin{equation}
    \label{eq:lagrelax}
    \begin{aligned}
    \LAG := \max_{\xv} \ & s \cdot \iprod{\cv^{(0)}}{\xv^{(0)}} + \sum_{i \in [s]} \iprod{\cv^{(i)}}{\xv^{(i)}} \\
    \text{s.t.}\ & (\xv^{(0)}, \xv^{(i)}) \in \conv(\mathcal{X}^{(i)}), \quad \forall i \in [s],
    \end{aligned}
\end{equation}
where
\[
\mathcal{X}^{(i)} :=
\left\{
(\yv^{(0)}, \yv^{(i)}) \;\middle\vert\;
\begin{array}{@{}l@{}}
A^{(0)} \yv^{(0)} + A^{(i)} \yv^{(i)} \leq \bv^{(i)}, \\
A^{(0)} \yv^{(0)} \leq \bv^{(j)}, \quad \forall j \in [s], \\
\yv^{(0)}, \yv^{(i)} \text{ binary}.
\end{array}
\right\}.
\]

We note that solving such a relaxation can be accomplished via standard column generation by solving multiple integer linear programs (ILPs) whose size is independent of the number of scenarios (more precisely, each has $2n$ variables and $2m$ linear constraints). Additionally, the Branch-and-Price algorithm only needs the optimal objective value of this relaxation, which can also be obtained by performing a Lagrangian relaxation of the nonanticipativity constraints. This latter method, which also uses multiple solves of ILPs whose size is independent of $s$, is usually the one employed in practice. See, for example,~\cite{boland2018combining} for more details. 

\paragraph{Branch-and-Price.} We consider the Branch-and-Price algorithm based on the dual decomposition relaxation. More precisely, we instantiate the elements of the Branch-and-Price framework as follows (see~\cite{lubbecke2010column,desaulniers2006column} for a general description of Branch-and-Price): 

\begin{enumerate}
    \item \textbf{Node relaxation:} In each node $N$ of the Branch-and-Price tree, we use the value $\LAG(N)$ of the dual decomposition relaxation \eqref{eq:lagrelax} associated with that node (i.e., with the appropriate variables $\xv^{(0)}$ fixed to $0/1$).

    \item \textbf{Branching rule:} Branch on any $\xv^{(0)}$ variable that has not yet been fixed at the node; it does not need to be fractional.

    \item \textbf{Node selection rule:} ``Best-bound,'' i.e., select the unpruned node with the largest computed dual value.

    \item \textbf{Finding an incumbent solution:} Whenever we are at a node $N$ where all variables $\xv^{(0)}$ are fixed to $0/1$ (call the fixing $\bar{\xv}^{(0)}$), we compute a solution $\bar{\xv}^{(1)},\ldots,\bar{\xv}^{(s)}$ such that $\bar{\xv}^{(0)}, \bar{\xv}^{(1)},\ldots,\bar{\xv}^{(s)}$ is feasible for the original problem and has value at least $\LAG(N)$.
\end{enumerate}

\section{Size of the Branch-and-Price Tree}
\label{sec:number_node}
In this section, we prove Lemma~\ref{lem:inform_gap_to_BB}, namely, that as long as the integrality gap of problem~\eqref{eq:IP} is 
$O_m\!\left( \frac{s \log s \log^2 n}{n} \right)$ with probability at least $1-\delta$, then with probability at least $1 - \delta - \frac{1}{n}$, the Branch-and-Price algorithm solves~\eqref{eq:IP} after exploring at most $n^{O_m(\log s)}$ nodes. (In Section~\ref{sec:int_gap}, we prove this assumption on the integrality gap, thereby obtaining an unconditional bound on the size of the Branch-and-Price tree.)

The high-level idea for proving this result is as follows:

\begin{enumerate}
    \item We consider the Lagrangian formulation (using an optimal dual solution) of the standard LP relaxation, and show that the primal optimal solution associated with each node in our Branch-and-Price (BP) tree has Lagrangian value at least $\OPT$ (Lemma~\ref{lemma:bestBound}). Intuitively, this holds because the ``best-bound'' rule guarantees that we never explore nodes with dual value below $\OPT$, and the Lagrangian value is at least as high as our dual bound from \eqref{eq:lagrelax}.

    \item We can associate with each node in our BP tree a \textbf{binary} vector with Lagrangian value at least $\OPT$, and at most $n$ internal nodes are associated with the same vector. This implies an upper bound on the size of our BP tree based on the number of such binary vectors with high Lagrangian value (Lemma~\ref{lemma:mappingOmega}).

    This is accomplished by appropriately rounding the primal optimal solution at each node, using the fact that the Lagrangian value is linear once the Lagrangian dual has been fixed.

    \item We prove that, with high probability, there are only $n^{O_m(\log s)}$ binary vectors with such high Lagrangian value, as long as the integrality gap is small.
\end{enumerate}

While the overall argument is similar to that of~\cite{dey2021branch}, the key
difference is that the original bound is exponential in the number of
constraints. In our setting, this number is $m \cdot s$, which would lead to an
exponential dependence on the number of scenarios. By contrast, our bound
scales as $n^{O_m(\log s)}$.
To obtain this improved dependence on $s$, we work only with the projection onto
the first-stage variables $\xv^{(0)}$, that is, we essentially restrict
attention to the matrix $A^{(0)}$, which has only $m$ rows. This approach
exploits the fact that the relaxation~\eqref{eq:lagrelax} convexifies each
scenario jointly with the first-stage variables, and is therefore stronger than
the LP relaxation considered in~\cite{dey2021branch}.

\bigskip

To formalize this argument, consider the standard LP relaxation of the full problem:
\begin{subequations}
\label{prob_lp_relax}
\begin{align}
\OPT_{\text{LP}} := \max & s \cdot \iprod{\cv^{(0)}}{\xv^{(0)}} + \sum_{i \in [s]} \iprod{\cv^{(i)}}{\xv^{(i)}} \\
\label{eq:sub_LP_relax}
\text{s.t.}\ & A^{(0)} \xv^{(0)} + A^{(i)} \xv^{(i)} \leq \bv^{(i)},\quad \forall i \in [s], \\
& \textbf{0} \leq \xv \leq \textbf{1}. 
\end{align}
\end{subequations}

Define the integrality gap by
$$\Delta := \OPT_{\text{LP}} - \OPT.$$
Consider the Lagrangian relaxation of this linear program:
\begin{align}
   \OPT_{\text{LP}} = & \min_{\mu \geq 0} \max_{0 \leq \xv \leq 1} \left(s \cdot \iprod{\cv^{(0)}}{\xv^{(0)}} + \sum_{r \in [s]} \iprod{\cv^{(r)}}{\xv^{(r)}}\right) + \left( \sum_{r \in [s]} \iprod{\mu^{(r)}}{\bv^{(r)} - A^{(0)} \xv^{(0)} - A^{(r)} \xv^{(r)}} \right) \label{eq:lag} \\
    = & \min_{\mu \geq 0} \max_{0 \leq \xv \leq 1} s \cdot \iprod{\cv^{(0)} - (A^{(0)})^{\top} \bigg(\frac{1}{s} \sum_r \mu^{(r)}\bigg)}{\xv^{(0)}} + \sum_{r \in [s]} \iprod{\cv^{(r)} - (A^{(r)})^{\top} \mu^{(r)}}{\xv^{(r)}} + \sum_{r \in [s]} \iprod{\mu^{(r)}}{\bv^{(r)}}. \notag
\end{align}

Let $(\xv_*,\mu_*)$ be the saddle point of this Lagrangian. To simplify the notation, let $\overline{\mu}_* := \frac{1}{s} \sum_{i \in [s]} \mu^{(i)}_*$. Given any first-stage point $\xv^{(0)}$, let
$$\Gamma(\xv^{(0)}) := s \cdot \iprod{\cv^{(0)} - (A^{(0)})^{\top} \overline{\mu}_*}{\xv^{(0)}} + \underbrace{\sum_{r \in [s]} \iprod{\cv^{(r)} - (A^{(r)})^{\top} \mu_*^{(r)}}{\xv^{(r)}_*} + \sum_{r \in [s]} \iprod{\mu_*^{(r)}}{\bv^{(r)}}}_{\text{constant}}$$
be the ``Lagrangian value'' of this point; note that we are using the optimal dual $\mu_*$ and also the optimal primal $\xv^{(r)}_*$ for all the scenarios. Define the set of first-stage binary solutions with $\Gamma$-value at least $\OPT$ as (recall that $\OPT$ is the integer optimum)
\begin{align*}
    \Omega := \{ \xv^{(0)} \in \{0,1\}^n : \Gamma(\xv^{(0)}) \ge \OPT\}.
\end{align*}

The first claim is that for every node in the Branch-and-Price tree, the primal optimal solution for the relaxation at that node has $\Gamma$-value at least $\OPT$.

\begin{lemma} \label{lemma:bestBound}
    Consider any internal node in the Branch-and-Price tree, and let $\bar{\xv}^{(0)},\bar{\xv}^{(1)},\ldots,\bar{\xv}^{(s)}$ be the optimal primal solution of the relaxation \eqref{eq:lagrelax} at this node. Then $\Gamma(\bar{\xv}^{(0)}) \ge \OPT$.
\end{lemma}

\begin{proof}
Let $\LAG(N)$ denote the exact optimal value of the relaxation
\eqref{eq:lagrelax} associated with node $N$ in the BP tree.
During the execution of BP, the value $\LAG(N)$ is used as the
node-selection criterion under the best-bound rule.

We claim that every internal node $N$ explored by BP satisfies
$\LAG(N) \ge \OPT.$
By contradiction, suppose there exists a node $N$ with
$\LAG(N) < \OPT$ that is selected for branching. By the best-bound
rule, this implies that for every other open node $N'$ in the BP tree, $\LAG(N') \le \LAG(N) < \OPT$.
On the other hand, there exists at least one node $N^*$ (open or pruned) whose
relaxation contains the globally optimal integral solution of the original
problem. For such a node, we have
$\OPT \le \LAG(N^*).$
Therefore, $N^*$ cannot be an open node at the time $N$ is selected, since all
open nodes have $\LAG(\cdot) < \OPT$. Hence, $N^*$ must already be
pruned.
If $N^*$ was pruned by integrality, then an incumbent solution of value $\OPT$
has already been found. If $N^*$ was pruned by bounds, then an incumbent of
value at least $\OPT$ must have been identified prior to pruning. In either
case, the current incumbent value is $\OPT$.
However, since $\LAG(N) < \OPT$, node $N$ would be pruned by the
bounding rule and thus could not be selected for branching, contradicting our
assumption. We conclude that every internal node of the BP tree satisfies
$\LAG(N) \ge \OPT$.

To conclude the proof, we show that for any node $N$ in the BP tree, the optimal
primal solution
$\bar{\xv}^{(0)}, \bar{\xv}^{(1)}, \ldots, \bar{\xv}^{(s)}$
of its relaxation~\eqref{eq:lagrelax} has $\Gamma$-value at least $\LAG(N)$, namely,
\begin{align}
\Gamma(\bar{\xv}^{(0)})
\;\ge\;
s \cdot \iprod{\cv^{(0)}}{\bar{\xv}^{(0)}}
+ \sum_{r \in [s]} \iprod{\cv^{(r)}}{\bar{\xv}^{(r)}} .
\label{eq:gammaBB}
\end{align}
To prove this, we expand the definition of $\Gamma$ and use the
fact that $\xv_*^{(r)}$ maximizes
$\iprod{\cv^{(r)} - (A^{(r)})^{\top} \mu_*^{(r)}}{\xv^{(r)}}$ to obtain
\begin{align}
\Gamma(\bar{\xv}^{(0)})
&\ge
s \cdot \iprod{\cv^{(0)} - (A^{(0)})^{\top} \overline{\mu}_*}{\bar{\xv}^{(0)}}
+ \sum_{r \in [s]} \iprod{\cv^{(r)} - (A^{(r)})^{\top} \mu_*^{(r)}}{\bar{\xv}^{(r)}}
+ \sum_{r \in [s]} \iprod{\mu_*^{(r)}}{\bv^{(r)}} \notag\\[4pt]
&=
s \cdot \iprod{\cv^{(0)}}{\bar{\xv}^{(0)}}
+ \sum_{r \in [s]} \iprod{\cv^{(r)}}{\bar{\xv}^{(r)}}
- \sum_{r \in [s]} \iprod{\mu_*^{(r)}}{A^{(0)} \bar{\xv}^{(0)} + A^{(r)} \bar{\xv}^{(r)}}
+ \sum_{r \in [s]} \iprod{\mu_*^{(r)}}{\bv^{(r)}}. \label{eq:gammaNode}
\end{align}
Since the solution
$\bar{\xv}^{(0)}, \bar{\xv}^{(1)}, \ldots, \bar{\xv}^{(s)}$
is feasible for the original constraints, i.e.,
\[
A^{(0)} \bar{\xv}^{(0)} + A^{(r)} \bar{\xv}^{(r)} \le \bv^{(r)}
\quad \text{for all } r \in [s],
\]
and $\mu_*^{(r)} \ge 0$, the last two terms of \eqref{eq:gammaNode} sum to a nonnegative value. Therefore,
\[
\Gamma(\bar{\xv}^{(0)})
\ge
s \cdot \iprod{\cv^{(0)}}{\bar{\xv}^{(0)}}
+ \sum_{r \in [s]} \iprod{\cv^{(r)}}{\bar{\xv}^{(r)}},
\]
which proves~\eqref{eq:gammaBB}. This completes the proof of the lemma.
\end{proof}

Given this result, the point $\bar{\xv}^{(0)}$ in the lemma would belong to the set $\Omega$ if $\bar{\xv}^{(0)}$ had only $0/1$ coordinates. However, this point is only an optimal solution of the relaxation \eqref{eq:lagrelax} (at some node) and thus may have fractional coordinates. Since $\Gamma(\cdot)$ is a linear function, we can round this point to a $0/1$ point with at least as large a $\Gamma$-value. More precisely, for a node $N$ of our BP tree, define $\xv^{(0)}_N$ as the argmax
\begin{align*}
   \xv^{(0)}_N := \argmax\Big\{ \Gamma(\xv^{(0)}) \,:\, \xv^{(0)} \in \{0,1\}^n \textrm{ and $\xv^{(0)}$ satisfies the variable fixings of the BP tree at $N$} \Big\}\,.
\end{align*}

We now have that these solutions $\xv^{(0)}_N$ belong to $\Omega$, and using this we can bound the size of our BP tree by the size of $\Omega$.

\begin{lemma} \label{lemma:mappingOmega}
    For every internal node $N$ in the BP tree, $\xv_N^{(0)} \in \Omega$. Moreover, for each $\xv^{(0)} \in \Omega$, there are at most $n$ internal nodes $N$ in the BP tree such that $\xv_N^{(0)} = \xv^{(0)}$. In particular, the size of the BP tree is at most $2n \cdot |\Omega| + 1$.
\end{lemma}

\begin{proof}
For the first part: for an internal node $N$ in the BP tree, the solution $\bar{\xv}^{(0)}$ from Lemma~\ref{lemma:bestBound} is feasible for the argmax defining $\xv_N^{(0)}$, which implies $\Gamma(\xv_N^{(0)}) \ge \Gamma(\bar{\xv}^{(0)}) \ge \OPT$, and so $\xv_N^{(0)} \in \Omega$.

For the second property: if two internal nodes $N_1$ and $N_2$ are not on the same path of the BP tree, then by looking at the common ancestor of these nodes we observe that there is some coordinate $\xv^{(0)}_j$ that is fixed to $0$ in $N_1$ and $1$ in $N_2$, or vice versa; this implies that their points $\xv_{N_1}^{(0)}$ and $\xv_{N_2}^{(0)}$ are different (in coordinate $j$). Thus, all the internal nodes $N$ whose solutions $\xv_{N}^{(0)}$ map to the same point $\xv^{(0)}$ of $\Omega$ must lie along a single path, and thus there are at most $n$ of them (recall that our Branch-and-Price algorithm only branches on the $n$ first-stage variables $\xv^{(0)}$).

This implies that there are at most $n \cdot |\Omega|$ internal nodes in the BP tree, which further implies at most $2n \cdot |\Omega| + 1$ total nodes (since it is a binary tree). This concludes the proof.
\end{proof}

In light of this result, in order to bound the size of the BP tree, it suffices to bound the number of points in $\Omega$.

\begin{lemma} \label{lemma:omega}
    Suppose that with probability at least $1 - \delta$, $\Delta \le O_m(\frac{s \log s \log^2 n}{n})$. Then with probability at least $1 - \frac{1}{n} - \delta$, we have $|\Omega| \le n^{O_m(\log s)}$.
\end{lemma}

Lemmas~\ref{lemma:mappingOmega} and~\ref{lemma:omega} thus imply the desired Lemma~\ref{lem:inform_gap_to_BB}.


\subsection{Proof of Lemma \ref{lemma:omega}}



Recall that $(\xv_*,\mu_*)$ is the saddle point of the Lagrangian \eqref{eq:lag}, and that $\overline{\mu}_* = \frac{1}{s} \sum_r \mu_*^{(r)}$.

Recalling the definition of $\Omega$, we would like to bound the number of $0/1$ points $\xv^{(0)}$ such that $\Gamma(\xv^{(0)}) \ge \OPT$. For that, consider the following hyperplane defined by the dual $\bar{\mu}_*$:
$$H := \{y \in \R^{m+1} : \iprod{(1, -\bar{\mu}_*)}{y} = 0\}.$$
The relevance of this hyperplane is that it transfers our task into a geometric one, since, as we show, $\Gamma(\xv^{(0)})$ can be related to the distance from the first-stage columns $(c^{(0)}_j, \av^{(0)}_j)$ to the hyperplane $H$ for the coordinates $j$ where $\xv^{(0)}$ and $\xv_*^{(0)}$ disagree. Let $d(y,H)$ denote the Euclidean distance between a point $y$ and the hyperplane $H$.

\begin{lemma} \label{lemma:dist}
    For every $\xv^{(0)} \in \{0,1\}^n$ we have
    $$\Gamma(\xv^{(0)}) \le \OPT + \Delta - s \cdot \sum_j d\Big(\big(c^{(0)}_j, \av^{(0)}_j\big), H\Big) \cdot \ones(x^{(0)}_j \neq x^{(0)}_{*,j}).$$
    Consequently, if $\xv^{(0)}$ belongs to $\Omega$, then
    $$\sum_j d\Big(\big(c^{(0)}_j, \av^{(0)}_j\big), H\Big) \cdot \ones(x^{(0)}_j \neq x^{(0)}_{*,j}) \,\le\, \frac{\Delta}{s}.$$
\end{lemma}

\begin{proof}
    The ``Consequently'' part follows directly from the definition of $\Omega$, so we only prove the first part of the lemma.

    Since $(\xv_*,\mu_*)$ is a saddle point for the Lagrangian \eqref{eq:lag}, by strong duality it achieves the value of $\OPT_{\LP}$, and so
    \begin{align}
        \OPT_{\LP} - \Gamma(\xv^{(0)}) &= s \cdot \iprod{\cv^{(0)} - (A^{(0)})^{\top} \overline{\mu}_*}{\xv^{(0)}_* - \xv^{(0)}} \notag\\
        &= s \cdot \sum_{j = 1}^n \Big(\cv^{(0)} - (A^{(0)})^{\top} \overline{\mu}_*\Big)_j \cdot (x^{(0)}_{*,j} - x^{(0)}_j). \label{eq:preDist}
    \end{align}
    Moreover, using the fact that $\xv^{(0)}_*$ is a saddle point and $\xv^{(0)}$ is a $0/1$ vector, we claim that for every $j$ we have
    \begin{align}
        \Big(\cv^{(0)} - (A^{(0)})^{\top} \overline{\mu}_*\Big)_j \cdot (x^{(0)}_{*,j} - x^{(0)}_j) \,=\, \Big|\Big(\cv^{(0)} - (A^{(0)})^{\top} \overline{\mu}_*\Big)_j\Big| \cdot \ones(x^0_j \neq x^0_{*,j}). \label{eq:absCol}
    \end{align}
    This is because if $\big(\cv^{(0)} - (A^{(0)})^{\top} \overline{\mu}_*\big)_j > 0$, then $x^{(0)}_{*,j} = 1$ (since it is a saddle point) and so $(x^{(0)}_{*,j} - x^{(0)}_j) = \ones(x^0_j \neq x^0_{*,j})$, whereas if $\big(\cv^{(0)} - (A^{(0)})^{\top} \overline{\mu}_*\big)_j < 0$, then $x^{(0)}_{*,j} = 0$ and $- (x^{(0)}_{*,j} - x^{(0)}_j) = \ones(x^{(0)}_j \neq x^{(0)}_{*,j})$.

    Finally, we claim that
    \begin{align}
        d\Big(\big(c^{(0)}_j, \av^{(0)}_j\big), H\Big) \,\le\, \Big|\Big(\cv^{(0)} - (A^{(0)})^{\top} \overline{\mu}_*\Big)_j\Big|. \label{eq:absCol2}
    \end{align}
    This is because the point $\left(\iprod{\overline{\mu}_*}{\av^{(0)}_j}, \av^{(0)}_j\right)$ belongs to $H$, and so it upper-bounds the distance:
    \begin{align*}
        d\left(\left(c^{(0)}_j, \av^{(0)}_j\right), H\right)
        &\le \Big\|\left(c^{(0)}_j, \av^{(0)}_j\right) - \left(\iprod{\overline{\mu}_*}{\av^{(0)}_j}, \av^{(0)}_j\right)\Big\|_2 \\
        &= \Big|c^{(0)}_j - \iprod{\overline{\mu}_*}{\av^{(0)}_j}\Big|
        = \Big|\big(\cv^{(0)} - (A^{(0)})^{\top} \overline{\mu}_*\big)_j\Big|.
    \end{align*}

    Combining \eqref{eq:absCol} and \eqref{eq:absCol2} and applying them to \eqref{eq:preDist}, we get
    \begin{align*}
        \OPT_{\LP} - \Gamma(\xv^{(0)}) \ge s \cdot \sum_{j = 1}^n d\Big(\big(c^{(0)}_j, \av^{(0)}_j\big), H\Big) \cdot \ones(x^0_j \neq x^0_{*,j}).
    \end{align*}
    Using the definition of the integrality gap $\Delta = \OPT_{\LP} - \OPT$ and rearranging the terms then gives the claim of the lemma.
\end{proof}

Roughly speaking, the previous lemma means that if $\xv^{(0)} \in \Omega$, then the vector $\xv^{(0)}$ and the saddle point $\xv^{(0)}_*$ can differ only on coordinates $j$ for which the column $\big(c^{(0)}_j, \av^{(0)}_j\big)$ is close to the hyperplane $H$. We now show that, with high probability, there are only ``a few'' of these columns; this in turn will imply that there are only ``a few'' possibilities for the vectors in $\Omega$ to differ from $\xv^{(0)}_*$, which upper-bounds the size of $\Omega$ as desired.

To make this precise, we define ``buckets'' of columns based on their distances to $H$: for $\ell \ge 1$ define
\begin{align*}
    J_\ell := \bigg\{j : d\Big(\big(c^{(0)}_j, \av^{(0)}_j\big),H\Big) \textrm{ is in the interval } \Big(\tfrac{\log n}{n} 2^\ell, \tfrac{\log n}{n} 2^{\ell+1} \Big] \bigg\},
\end{align*}
and define $J_{\mathrm{rem}} := [n] \setminus \bigcup_{\ell \ge 1} J_\ell$ as the remaining columns (i.e., those with distance at most $\frac{\log n}{n} \cdot 2$). The next lemma establishes the connection between the number of column elements in these buckets and the size of $\Omega$. The proof is the same as that of Lemma~5 of~\cite{dey2021branch}, but we reproduce it for completeness. We use the notation $\binom{u}{\le v} := \binom{u}{0} + \ldots + \binom{u}{v}$ for positive integers $u \ge v$.

\begin{lemma} \label{lemma:count}
    We have the following upper bound:
    \begin{align*}
        |\Omega| \le 2^{|J_{\mathrm{rem}}|} \cdot \prod_{\ell=1}^{\log C}
        \binom{|J_{\ell}|}{\le C/2^\ell},
    \end{align*}
    where $C := \frac{n}{\log n} \cdot \frac{\Delta}{s}$.
\end{lemma}

\begin{proof}
    First, observe that for every column $j$ in $\bigcup_{\ell \ge 1} J_\ell$, the coordinate $x^{(0)}_{*,j}$ has value $0$ or $1$: these columns have strictly positive distance to $H$, hence from \eqref{eq:absCol2} they have nonzero reduced costs $\big(\cv^{(0)} - (A^{(0)})^{\top} \overline{\mu}_*\big)_j$, and so the saddle-point solution $x^{(0)}_{*,j}$ is at the boundary of its domain $[0,1]$.

    Next, notice that every point in $\Omega$ can be thought of as being created by starting with the vector $\xv^{(0)}_*$ and then changing some of its coordinates, and because of Lemma~\ref{lemma:dist} we:
\begin{itemize}
    \item cannot change the value of $\xv^{(0)}_*$ in any coordinate $j$ in a bucket $J_\ell$ with $\ell > \log C$;
    \item can only flip the value of $\xv^{(0)}_*$ in at most $\frac{C}{2^\ell}$ of the coordinates $j \in J_\ell$, for each $\ell = 1, \dots, \log C$ (recall that both $x^{(0)}_{*,j}$ and $x^{(0)}_j$ take only the values $0$ or $1$);
    \item may set a new arbitrary $0/1$ value for, in principle, all coordinates in $J_{\mathrm{rem}}$.
\end{itemize}

Since there are at most $2^{|J_{\mathrm{rem}}|} \cdot \prod_{\ell=1}^{\log C} \binom{|J_\ell|}{\leq C/2^\ell}$ possibilities in this process, we have the desired upper bound.
\end{proof}

Since both the columns $(c^{(0)}_j, \av^{(0)}_j)$ and the hyperplane are random objects that depend on the entries of the input problem, the sizes of the sets $J_\ell$ are also random. In order to control these sizes with high probability, we will use the following technical lemma (Theorem~4 of~\cite{dey2021branch}); let $\mathbb{S}^{k-1}$ denote the unit sphere in $\R^k$.

\begin{lemma}[Theorem 4 of~\cite{dey2021branch}] \label{lemma:uniformSlab}
    For $u \in \mathbb{S}^{k-1}$ and $w \ge 0$, define the \emph{slab} of normal $u$ and width $w$ as $$S_{u,w} := \Big\{y \in \R^k : \iprod{u}{y} \in [-w,w] \Big\}.$$ Let $Y^1,\ldots,Y^n$ be independent random vectors uniformly distributed in the cube $[0,1]^k$. For $n \ge k$, then with probability at least $1 - \frac{1}{n}$, we have that for all $u \in \mathbb{S}^{k-1}$ and $w \ge \frac{\log n}{n}$, at most $60 nwk$ of the $Y^j$'s belong to $S_{u,w}$.\footnote{We note that the assumption $n \ge k$ is without loss of generality, since the result is trivially true otherwise: the assumption $w \ge \frac{\log n}{n}$ implies that the final upper bound is $60 n w k \ge 60 \log n \cdot k > k$, making the statement vacuous when $n \le k$.}
\end{lemma}
Notice that if a column $\left(\cv^{(0)}_j, \av^{(0)}_j\right)$ is such that $j$ belongs to $J_\ell$, then this column belongs to the $(m+1)$-dimensional slab $\left\{y : d(y,H) \le \tfrac{\log n}{n} 2^{\ell+1}\right\}$ of width $\frac{\log n}{n} 2^{\ell+1}$; thus, the previous lemma can be used to upper-bound the size $|J_\ell|$ (and similarly for $|J_{\mathrm{rem}}|$). Notice that we do not even need to take a union bound over the $\ell$'s, since the previous lemma controls their sizes simultaneously. This gives the following:

\begin{cor} \label{cor:J}
    With probability at least $1 - \frac{1}{n}$, we have simultaneously
    \begin{gather*}
        |J_{\mathrm{rem}}| \le 120 (m + 1) \log n,\\
        |J_\ell| \le 60(m+1) 2^{\ell + 1} \log n,~~~~\forall \ell \in [\log n - 1].
    \end{gather*}
\end{cor}

Now we are ready to prove Lemma~\ref{lemma:omega}.


\begin{proof}[of Lemma \ref{lemma:omega}]
    Since we assumed the integrality gap satisfies $\Delta \le O_m(\frac{s \log s\log^2 n}{n})$ with probability at least $1 - \delta$, by taking a union bound, with probability at least $1 - \frac{1}{n} - \delta$ both this event and the upper bounds on $|J_\ell|$ and $|J_{\mathrm{rem}}|$ from Corollary~\ref{cor:J} hold. We show that in this case $|\Omega| \le n^{O_m(\log s)}$, proving the lemma.

    Under this event, we have $C := \frac{n}{\log n} \cdot \frac{\Delta}{s} \le f(m) \cdot \log s \log n$ for some increasing function $f(\cdot) \ge 1$. Let $\ell_0 \ge 0$ be a sufficiently large integer of order $\frac{\log(O_m(\log s))}{2}$ but satisfying $60 (m+1) 2^{2\ell_0 + 1} \log n \ge 4 C$ (which exists by the previous upper bound on $C$). We will use the following standard estimate: $\binom{a}{\le b} \le (\frac{4a}{b})^b$ whenever $a \ge 4b$ (our definition of $\ell_0$ ensures that this requirement holds in our application of the estimate). Then we have the following upper bound:
    \begin{align*}
        \prod_{\ell=\ell_0}^{\log C}
        \binom{|J_{\ell}|}{\le C/2^\ell} \le \prod_{\ell=\ell_0}^{\log C} \bigg(\frac{240 (m+1)2^{\ell+1} \log n}{(f(m) \cdot \log s \log n)/2^\ell} \bigg)^{C/2^\ell} &\le \prod_{\ell=\ell_0}^{\log C} \bigg(\frac{O_m(1) \cdot 2^{2\ell}}{\log s} \bigg)^{C/2^\ell}\\
        &\le \bigg(\frac{O_m(1)}{\log s} \bigg)^{C \sum_{\ell \ge 1} \frac{1}{2^\ell}} \cdot 2^{2C \sum_{\ell \ge 1} \frac{\ell}{2^\ell}}\\
        &\le \bigg(\frac{O_m(1)}{\log s} \bigg)^{O(C)} \cdot 2^{O_m(C)} \le n^{O_m(\log s)}.
    \end{align*}
    For the terms with $\ell < \ell_0$, we can use the crude estimate $\binom{a}{b} \le 2^a$ to get
    \begin{align*}
        \prod_{\ell=1}^{\ell_0 - 1}
        \binom{|J_{\ell}|}{\le C/2^\ell} \,\le\, \prod_{\ell = 1}^{\ell_0 - 1} 2^{|J_\ell|} \,\le\, 2^{O_m(1) \cdot \ell_0 \cdot 2^{\ell_0} \log n} \,\le\, 2^{O_m(1)\cdot \log \log s \cdot \sqrt{\log s} \cdot \log n} \,\le\, n^{O_m(\log s)}.
    \end{align*}
    Finally, directly from Corollary~\ref{cor:J} we have $|J_{\mathrm{rem}}| \le O_m(\log n)$. Employing these bounds in Lemma~\ref{lemma:count} we obtain
    \begin{align*}
        |\Omega| \,\le\, 2^{|J_{\mathrm{rem}}|} \cdot \prod_{\ell=1}^{\log C}
        \binom{|J_{\ell}|}{\le C/2^\ell} \,\le\, n^{O_m(\log s)}.
    \end{align*}
   This concludes the proof of Lemma~\ref{lemma:omega}.
\end{proof}




\section{Integrality Gap of Random Two-stage Integer Programs}
\label{sec:int_gap}

In this section, we prove Theorem~\ref{thm:inform_int_gap}. Our approach extends the probabilistic analysis of \cite{dyer1989probabilistic}, which was developed for single-block packing integer programs, to a broader class of multi-block packing integer programs with a two-stage structure.  The main technical difficulty, as well as a detailed comparison with the argument of \cite{dyer1989probabilistic}, are deferred to Section~\ref{sec:steps_of_prf_int_gap}, where the necessary background has been introduced.

Recall that the natural LP relaxation of \eqref{eq:IP}, parameterized by the right-hand side vector $\bv$, is given by
\begin{subequations}
\label{prob_lp_relax}
\begin{align}
\OPT_{\mathrm{LP}}(\bv)
:=\max \ & s \, \langle \cv^{(0)}, \xv^{(0)} \rangle
+ \sum_{i \in [s]} \langle \cv^{(i)}, \xv^{(i)} \rangle \\
\text{s.t.}\quad
& A^{(0)} \xv^{(0)} + A^{(i)} \xv^{(i)} \le \bv^{(i)}, 
\quad \forall i \in [s], \label{eq:sub_LP_relax}\\
& \mathbf{0} \le \xv \le \mathbf{1}. \label{eq:LPLast}
\end{align}
\end{subequations} and $\Delta$ denotes the (additive) integrality gap of this relaxation, namely $$\Delta := \OPT_{\text{LP}}(\bv)  - \OPT(\bv).$$
The formal statement of Theorem~\ref{thm:inform_int_gap} is as follows:
\begin{thm} \label{thm:smallGap}
    Let $\beta_{\min} := \min_{i \in [s], r \in [m]} \beta^{(i)}_r$ and set $\alpha := \beta_{\min} - \frac{1}{4}$.
    Under this stochastic model, and under the assumptions that $n \ge 2^{m+4} (m+2) \log(2mn)$, $2 \sqrt{n \ln s n m} + 6\alpha m \log n < n \cdot \big(\beta_{\min} - \frac{1}{4}\big), \frac{2\sqrt{3}m}{\alpha} \leq (m \log n)^{1/4}$, and $n$ is at least a sufficiently large constant, there exists a constant $\pi \in (0,1)$ that only depends on $m,\beta$ such that
    \begin{align*}
        \P\bigg(\Delta \le O_m\bigg(\frac{sf\log^2 n}{n}\bigg)\bigg) \ge 1 - {\frac{2s+3}{n^2}-(s+1) \pi^f}
    \end{align*}
    for all $f \geq \max\left\{\frac{\alpha 2^{m+6} \log (2mn)}{\log n}, \frac{2^{2m+8}}{\alpha^m}\right\}$.
    By choosing $f = \sigma \log s$ for some sufficiently large constant $\sigma$ (depends on $m,\beta$) and for sufficiently large $n$, this implies that
    \begin{align*}
         \P\bigg(\Delta \le O_m\bigg(\frac{s \log s\log^2 n}{n}\bigg)\bigg) \ge 0.9999.
    \end{align*}
\end{thm}

The full proof of this theorem is rather involved and we sketch the main proof ideas in Section~\ref{sec:steps_of_prf_int_gap}.

\subsection{Main Steps of Theorem~\ref{thm:smallGap}} 
\label{sec:steps_of_prf_int_gap}

Since the proof of Theorem~\ref{thm:smallGap} consists of several technical components, we first present the key lemmas and explain how they together imply the theorem. Each lemma is then proved in a separate subsection.

To bound the additive integrality gap 
$\Delta = \OPT_{\mathrm{LP}}(\bv) - \OPT(\bv)$,
we first derive a bound on how the objective value changes when moving from an optimal LP solution to an integral one (e.g., via rounding). In later sections, we show how to construct such a desired rounding.
To get a hold of the LP relaxation, it will be useful to consider its saddle-point (min–max) formulation similar to the previous argument. Let 
\begin{align}
L^{std}_{\bv}(\xv,\mu) := s \cdot \iprod{\cv^{(0)}}{\xv^{(0)}} + \sum_{i \in [s]} \iprod{\cv^{(i)}}{\xv^{(i)}} + \sum_{i \in [s]} \iprod{\mu^{(i)}}{\bv^{(i)} - A^{(0)} \xv^{(0)} - A^{(i)} \xv^{(i)} } \label{eq:lagStd}
\end{align}
denote the objective of this Lagrangian relaxation, so the LP value is $\OPT_{\textup{LP}} = \max_{0 \le \xv \le 1} \min_{\mu \ge 0} L^{std}_{\bv}(\xv,\mu)$.

We use $\beta_{\min} := \min_{i \in [s], r \in [m]} \beta^{(i)}_r$, and set $\alpha := \beta_{\min} - \frac{1}{4}$. With a lot of hindsight, define 
\begin{align}
d := \frac{3\alpha m \log n}{2} \cdot \bigg(1 + \frac{1}{\sqrt{3} \cdot  n^2}\bigg); \label{eq:defD}
\end{align}
We will actually consider optimal primal/dual solution for the LP with modified RHS $\bv_{\diamond} := \bv - 2d\ones$. This extra slack will be necessary {to ensure feasibility} when we perform our rounding in Lemma \ref{lemma:condGap}, in Section \ref{sec:goodRounding}. 
The following lemma is used to control integrality gap $\Delta$.  
Here, $\av^{(i)}_j$ denotes the $j$-th column of the matrix $A^{(i)}$ and $\xv_j^{(i)}$ denotes the $j$-th entry of the vector $\xv^{(i)}$.
\begin{lemma} \label{lemma:startBoundGap}
    Let $(\xv_{\diamond}, \mu_{\diamond})$ be saddle-point for $L^{std}_{\bv_{\diamond}}(\xv,\mu)$ (notice the modified right-hand side $\bv_{\diamond}$); define $\mu^{(0)} := \frac{1}{s} \sum_{i=1}^s \mu^{(i)}$. Also consider any feasible solution $\bar{\xv}$ for the original (integer) problem \eqref{eq:IP} that keeps all $1$-entries of $\xv^{(i)}_{\diamond}$ (i.e., $\xv^{(i)}_{\diamond j} = 1 \implies \bar{\xv}^{(i)}_j = 1$).
    Then, the additive integrality gap of the original problem can be upper bounded as 
    $\Delta \le s \cdot \Delta^{(0)}(\bar{\xv}^{(0)}) + \sum_{i = 1}^s \Delta^{(i)}(\bar{\xv}^{(i)})$, where for all $i = 0, \ldots, s$ $$\Delta^{(i)}(\bar{\xv}^{(i)}) := \big\|\mu^{(i)}_{\diamond}\big\|_1 \cdot \underbrace{\big\|d \ones - A^{(i)} (\xv^{(i)}_{\diamond} - \bar{\xv}^{(i)})\big\|_{\infty}}_{\text{constraint slackness}} + \sum_{j : \bar{\xv}^{(i)}_j > \xv^{(i)}_{\diamond j}} \underbrace{\Big(\langle \mu^{(i)}_{\diamond}, \colij \rangle - c^{(i)}_j
 \Big).}_{\text{reduced cost}}$$
\end{lemma}
The proof of this lemma follows the same line of argument as in
\cite{dyer1989probabilistic,borst2023integrality}, and is included in
Section~\ref{sec:firstBound} for completeness.
We emphasize that controlling the change in the objective value
$\langle \cv, \xv \rangle$ alone is not sufficient, since the LP relaxation is
solved with a modified right-hand side $\bv_{\diamond}$.
In order to relate the rounded solution back to the original problem with
right-hand side $\bv$, we require an upper bound on
$\OPT_{\mathrm{LP}}(\bv)$.
This bound is obtained via a standard sensitivity analysis argument based on
the optimal dual multiplier $\mu_{\diamond}$.
Taken together, Lemma~\ref{lemma:startBoundGap} decomposes the integrality gap
into two components: one arising from reduced costs and the other from
constraint slackness.

\paragraph{Proof strategy of Theorem~\ref{thm:smallGap}.} Our proof strategy follows a line similar to that of \cite{dyer1989probabilistic}.
In the single-block case, i.e., when $s=1$, and motivated by
Lemma~\ref{lemma:startBoundGap}, the construction of a feasible integral
solution $\bar{\xv}$ that yields the desired bound on $\Delta$ proceeds as
follows.
We first (conceptually) solve $\OPT_{\mathrm{LP}}(\bv_{\diamond})$ to obtain an
optimal solution $\xv_{\diamond}$, and then round down all fractional entries
of $\xv_{\diamond}$.
While this rounding step guarantees integrality, it inevitably increases the
constraint slackness.
To compensate for this effect, we subsequently flip a carefully selected
subset of originally zero variables with small reduced cost in order to absorb
the additional slackness introduced by rounding.
A key difficulty of this approach is that the gap formula in
Lemma~\ref{lemma:startBoundGap} depends jointly on
$(\xv_{\diamond}, \mu_{\diamond})$, whose distributions are complicated and
correlated with the random data $(A,\cv)$ in~\eqref{eq:IP}.
Therefore, the existence of a sufficiently large collection of such zero
variables is highly non-trivial and is guaranteed via a sophisticated
probabilistic argument.

\paragraph{Comparison to \cite{dyer1989probabilistic}.} In this work, we extend the result of \cite{dyer1989probabilistic} to
multi-block packing integer programs with a two-stage structure.
If one ignores the block structure of~\eqref{eq:IP} and applies the argument of
\cite{dyer1989probabilistic} directly, one obtains $
\Delta \le O_m\!\left(\exp(\mathrm{poly}(s)) \frac{\log^2 n}{n}\right)$
with high probability, which is significantly weaker than the bound stated in
Theorem~\ref{thm:smallGap}.
An important observation is that if we fix the optimal dual variable
$\mu_{\diamond}$, then all block variables $\xv^{(i)}$ become independent of
each other in the Lagrangian
$L^{std}_{\bv_{\diamond}}(\xv,\mu)$, since all coupling constraints have been
dualized.
This phenomenon is also reflected in the gap formula in
Lemma~\ref{lemma:startBoundGap}, namely, there are no cross terms between
different block variables $\xv^{(i)}$.
In this case, a natural attempt is to apply the argument of
\cite{dyer1989probabilistic} to each block separately, which hopefully would yield $
\Delta \le O_m\!\left( s \log s \frac{\log^2 n}{n}\right),$
as promised in Theorem \ref{thm:smallGap}.
However, this direct approach encounters a technical difficulty.
The analysis in \cite{dyer1989probabilistic,borst2023integrality} depends critically on the magnitude
of the optimal dual solution and on the number of zero entries in the optimal
primal solution.
Both quantities are sensitive to the block structure of~\eqref{eq:IP} and to the
number of scenarios $s$.
In fact, using the arguments in
\cite{dyer1989probabilistic,borst2023integrality}, one can show that the optimal
dual solution and the number of zero entries of the optimal primal solution
behave as in the single-block case \textbf{on average}, that is,
$\normm{\frac{1}{s} \sum_{i=1}^s \mu^{(i)}_{\diamond}}_{\infty}
\approx
\normm{\text{optimal dual variables when } s=1}_{\infty}.$
Despite this favorable average behavior, this information alone is insufficient
to characterize the full distribution of these quantities.
Consequently, it does not suffice to reduce the analysis of the multi-block
problem~\eqref{eq:IP} to that of the single-block case.
Indeed, a naive blockwise application of the analysis in
\cite{dyer1989probabilistic} would still result in an exponential dependence on
$s$, yielding a bound of the form
$O_m\!\left(\exp(\mathrm{poly}(s)) \frac{\log^2 n}{n}\right).$
Our main contribution is to show that if we impose an additional lower-bound
requirement on the right-hand side $\bv$, namely,
$\bv^{(i)}_r = \beta^{(i)}_r \cdot 2n,
\beta^{(i)}_r \in \big(\tfrac{1}{4}, \tfrac{1}{2}\big),
\ \forall i \in [s],\ r \in [m],$
then the average behavior of the magnitude of the optimal dual solution and the
number of zero entries in the optimal primal solution accurately reflects their
behavior in each individual block.
Under this condition, we can apply the single-block analysis to each block and
obtain the desired bound stated in
Theorem~\ref{thm:smallGap}.

Let $\mathcal{N}_0^{(i)} := \{j \in [n] : \xv^{(i)}_{\diamond} = 0\}$. Now let $Good$ be the event that the following happens: 
\begin{enumerate}
    \item For every block $i \in [s] \cup \{0\}$, $\normm{\mu_{\diamond}^{(i)}}_{\infty} \leq \frac{1}{2\alpha}$.
    \item For every block $i \in [s] \cup \{0\}$, if 
    $\normm{\mu_{\diamond}^{(i)}}_1 \ge \frac{f \log n}{2 \alpha n}$, then  $\big|\mathcal{N}_0^{(i)}\big| \geq \frac{n}{2^{m+{2}}} \cdot \min\big\{ \frac{1}{4} \normm{\mu^{(i)}_\diamond}_1, 1\big\}$.
\end{enumerate}

Note that in Item~2 of the event \emph{Good}, we require the cardinality of
$\mathcal{N}_0^{(i)}$ to scale proportionally with
$\|\mu^{(i)}_\diamond\|_1$, unless $\|\mu^{(i)}_\diamond\|_1$ is already
bounded by a small number with order $O(\frac{\log n}{n})$. This requirement, while subtle, is essential for the
analysis. In later steps, we will sample columns from $\mathcal{N}_0^{(i)}$
to enforce certain structural properties needed in the proof of
Lemma~\ref{lemma:mainDeltai}, and the corresponding sampling procedure
depends explicitly on the magnitude of the dual vector
$\mu^{(i)}_\diamond$. The condition above guarantees that
$\mathcal{N}_0^{(i)}$ contains a sufficiently large collection of indices,
ensuring that enough samples can be drawn to carry out the rounding
argument successfully.

The following lemma is the main technical result of this section. It shows
that, conditioned on the event \emph{Good}, the integrality gap
$\Delta$ is at most $O_m\!\left(\frac{s f \log^2 n}{n}\right)$ with high
probability.

\begin{lemma} 
{Let $\pi$ be the same constant defined in Theorem \ref{thm:smallGap}}. Under the condition of Theorem \ref{thm:smallGap}, It follows that 
\label{lemma:condGap}
    $$\Pr\bigg(\Delta > \Omega_m\bigg(\frac{{sf}\log^2n}{n}\bigg) ~\bigg|~ Good\bigg) \le  {\frac{(s+1)}{n^2}+(s+1) \pi^f}.$$
\end{lemma}

This is the most technical lemma and the proof is deferred to Section~\ref{sec:goodRounding}. 

Now we need to show that the event $Good$ holds with high probability, i.e., that the complement event $Good^c$ happens with tiny probability. Notice that this complement event consists of two union of the events:
\begin{gather*}
    E1:~~\textrm{$\exists i \in [s] {\cup \{0\}}$ such that $\|\mu^{(i)}_{\diamond}\|_{\infty} > \frac{1}{2\alpha}$} \\
    E2:~~
    \begin{array}{l}
      \textrm{$\forall i \in [s] {\cup \{0\}}, \|\mu^{(i)}_{\diamond}\|_{\infty} \le \frac{1}{2\alpha}$ but} 
      \textrm{$\exists i \in [s] {\cup \{0\}}$ such that $\|\mu^{(i)}_{\diamond}\|_1 \ge \frac{f \log n}{2\alpha n}$ and $\big|\mathcal{N}_0^{(i)}\big| < \frac{n}{2^{m+{2}}} \cdot \min\Big\{ \frac{1}{4} \normm{\mu^{(i)}_\diamond}_1, 1\Big\}$}.  
    \end{array}
\end{gather*}

The next lemma bounds the probability of the first event $E1$. 

\begin{lemma} \label{lemma:boundDual}
    With probability at least $1 - \frac{1}{n^2}$, under the condition of Theorem \ref{thm:smallGap}, we have $\|\mu^{(i)}_{\diamond}\|_{\infty} \le  {\dfrac{1}{2\alpha}}$ for all blocks $i \in [s] \cup \{0\}$.
\end{lemma}

The second event $E2$ is controlled by the following lemma. 

\begin{lemma} \label{lemma:numZeros} Under the condition of Theorem \ref{thm:smallGap},
    the probability that $\forall i \in [s] {\cup \{0\}}, \|\mu^{(i)}_{\diamond}\|_{\infty} \le \frac{1}{2\alpha}$ but $\exists i \in [s]{\cup \{0\}}$ such that $\|\mu^{(i)}_{\diamond}\|_1 \ge \frac{f \log n}{2\alpha n}$ and $\big|\mathcal{N}_0^{(i)}\big| < \frac{n}{2^{m+{2}}} \cdot \min\big\{ \frac{1}{4} \normm{\mu^{(i)}_\diamond}_1, 1\big\}$ is at most ${(s+1)/n^2}$.
\end{lemma}

With these results at hand, the desired bound on the integrality gap directly follows:

\begin{proof}[~of Theorem \ref{thm:smallGap}]
    Combining Lemmas~\ref{lemma:boundDual} and \ref{lemma:numZeros}, we obtain that the complement event $Good^c$ holds with probability at most $\frac{s+2}{n^2}$. Therefore, further using Lemma~\ref{lemma:condGap}, the probability of a large integrality gap can be upper bounded as 
    \begin{align*}
        \Pr\big(\Delta > \Omega_m(\tfrac{{sf}\log^2n}{n})\big) &= \Pr\big(\Delta > \Omega_m(\tfrac{{sf}\log^2n}{n}) \mid Good\big) \Pr(Good) + \Pr\big(\Delta > \Omega_m(\tfrac{{sf}\log^2n}{n}) \mid Good^c\big) \Pr(Good^c)  \\
        &\le {\frac{s+1}{n^2} + (s+1)\pi^f}  + \frac{{s+2}}{n^2}  = {\frac{2s + 3}{n^2}+ (s+1)\pi^{f}}. 
    \end{align*}
    This concludes the proof of the theorem. 
\end{proof}

\subsection{Proof of Lemma \ref{lemma:startBoundGap}} \label{sec:firstBound}
Recall that the fact that $(\xv_{\diamond}, \mu_{\diamond})$ is a saddle point
of $L^{\mathrm{std}}_{\bv_{\diamond}}(\xv,\mu)$ means that
\[
\OPT_{\LP}(\bv_{\diamond})
= \max_{0 \le \xv \le 1} \min_{\mu \ge 0} L^{\mathrm{std}}_{\bv_{\diamond}}(\xv,\mu)
= \max_{0 \le \xv \le 1} L^{\mathrm{std}}_{\bv_{\diamond}}(\xv,\mu_{\diamond})
= \min_{\mu \ge 0} L^{\mathrm{std}}_{\bv_{\diamond}}(\xv_{\diamond},\mu),
\]
that is, $\xv_{\diamond}$ and $\mu_{\diamond}$ are optimal primal and dual
solutions, respectively.
Since $L^{\mathrm{std}}$ is bilinear and the feasible regions are defined by box
constraints, the following complementary slackness conditions hold.



\begin{enumerate}[label={(cs.\roman*)}]
\item For all $i \in [s]$, we have $\iprods{\mu_{\diamond}^{(i)}}{\bv^{(i)}_{\diamond} - A^{(0)} \xv^{(0)}_{\diamond} - A^{(i)} \xv^{(i)}_{\diamond}} = 0$ (if we had ``$<$'' then $\OPT_{\LP}(\bv_{\diamond}) = \min_{\mu \ge 0} L^{std}_{\bv_{\diamond}}(\xv_{\diamond},\mu)$ would be $-\infty$, which is impossible, and if we had ``$>$'' then $\mu_{\diamond}$ would not be optimal) \label{cs.i}. 
 \item For every $i,j$, if $c^{(i)}_j - \iprods{\mu_{\diamond}^{(i)}}{\colij} < 0$ then $\xv_{\diamond j}^{(i)} = 0$, and if $c^{(i)}_j - \iprods{\mu_{\diamond}^{(i)}}{\colij} > 0$ then $\xv_{\diamond j}^{(i)} = 1$ (both by optimality of $\xv_{\diamond j}^{(i)}$) \label{cs.ii}.
\end{enumerate}

Recall that the integrality gap is 
$\Delta := \OPT_{\mathrm{LP}}(\bv) - \OPT(\bv).$
The LP optimum can be upper bounded using the dual vector $\mu_{\diamond}$ as
$\OPT_{\mathrm{LP}}(\bv)
\le \max_{0 \le \xv \le 1} L^{\mathrm{std}}_{\bv}(\xv,\mu_{\diamond}).$
Moreover, $\xv_{\diamond}$ remains an optimal maximizer for this problem with
right-hand side $\bv$, since \ref{cs.ii} shows that the optimal choice of $\xv$
depends only on the reduced costs and is therefore independent of the
right-hand side.
Hence,
$
\OPT_{\mathrm{LP}}(\bv)
\le L^{\mathrm{std}}_{\bv}(\xv_{\diamond},\mu_{\diamond}).
$
On the other hand, the integer optimum $\OPT(\bv)$ is lower bounded by the value
$s \langle \cv^{(0)}, \bar{\xv}^{(0)} \rangle
+ \sum_{i \in [s]} \langle \cv^{(i)}, \bar{\xv}^{(i)} \rangle$
of the feasible integral solution $\bar{\xv}$.

Combining these bounds and adding and subtracting
$\sum_{i \in [s]}
\langle \mu^{(i)}_{\diamond},
A^{(0)} \bar{\xv}^{(0)} + A^{(i)} \bar{\xv}^{(i)} \rangle$,
we obtain
\begin{align}
\Delta
\le\;
&s \cdot
\big\langle \cv^{(0)} - (\mu^{(0)}_{\diamond})^{\top} A^{(0)},
\xv^{(0)}_{\diamond} - \bar{\xv}^{(0)} \big\rangle
+ \sum_{i \in [s]}
\big\langle \cv^{(i)} - (\mu^{(i)}_{\diamond})^{\top} A^{(i)},
\xv^{(i)}_{\diamond} - \bar{\xv}^{(i)} \big\rangle \notag \\
&\quad
+ \sum_{i \in [s]}
\big\langle \mu^{(i)}_{\diamond},
\bv^{(i)} - A^{(0)} \bar{\xv}^{(0)} - A^{(i)} \bar{\xv}^{(i)} \big\rangle .
\label{eq:startBoundGap}
\end{align}
We next bound the first term in \eqref{eq:startBoundGap}. Recall that $\av^{(0)}_j$ denotes the $j$th column of the matrix $A^{(0)}$, we have
\begin{align}
    \iprod{\cv^{(0)} - (\mu^{(0)}_{\diamond})^{\top} A^{(0)}\,}{\,\xv^{(0)}_{\diamond} -\bar{\xv}^{(0)}} &= \sum_j \Big(c^{(0)}_j - \iprod{\mu_{\diamond}^{(0)}}{\av^{(0)}_j}\Big)\cdot (\xv^{(0)}_{\diamond j} - \bar{\xv}^{(0)}_j) \notag\\
    &\le  \sum_{j : \bar{\xv}^{(0)}_j > \xv^{(0)}_{\diamond j}} \Big(\langle \mu^{(0)}_{\diamond}, \av^{(0)}_j\rangle - c^{(0)}_j
 \Big),  \label{eq:gapFirstTerm}
    \end{align}
    where we justify the contribution of each term $j$ to the last inequality as follows. First, when $\xv^{(0)}_{\diamond j} = 1$, by hypothesis we also have $\bar{\xv}^{(0)}_{j} = 1$, so $\xv^{(0)}_{\diamond j} - \bar{\xv}^{(0)}_j = 0$ and the term $j$ does not contribute to the sum. Now when $\xv_{\diamond j}^{(0)} < 1$, \ref{cs.ii} discussed above gives that $c^{(0)}_j - \iprods{\mu_{\diamond}^{(0)}}{\av^{(0)}_j} \le 0$, and so swapping the differences the term becomes $(\iprods{\mu_{\diamond}^{(0)}}{\av^{(0)}_j} - c^{(0)}_j)\cdot (\bar{\xv}^{(0)}_j - \xv^{(0)}_{\diamond j}) \le (\iprods{\mu_{\diamond}^{(0)}}{\av^{(0)}_j} - c^{(0)}_j)\cdot \bar{\xv}^{(0)}_j$. When $\bar{\xv}^{(0)}_j = 0$, the $j$th term is then non-positive and can be dropped, and when $\bar{\xv}^{(0)}_j = 1$ (which is precisely the case $\bar{\xv}^{(0)}_j > \xv^{(0)}_{\diamond j}$) the $j$th term can be upper bounded as $\iprods{\mu_{\diamond}^{(0)}}{\av^{(0)}_j} - c^{(0)}_j$, as desired. 

    The same argument allows us to bound the second term in \eqref{eq:startBoundGap} analogously as    
    \begin{align}
    \iprod{c^{(i)} - (\mu^{(i)}_{\diamond})^{\top} A^{(i)}\,}{\,\xv^{(i)}_{\diamond} - \bar{\xv}^{(i)}} \le  \sum_{j : \bar{\xv}^{(i)}_j > \xv^{(i)}_{\diamond j}} \Big(\langle \mu^{(i)}_{\diamond}, \av^{(i)}_j\rangle - c^{(i)}_j
 \Big).  \label{eq:gapSecondTerm}
    \end{align}
    To bound the last term of \eqref{eq:startBoundGap}, by \ref{cs.i} discussed above we have that that related term, but with respect to $\xv_{\diamond}$, is actually 0, namely $\iprods{\mu_{\diamond}^{(i)}}{\bv^{(i)}_{\diamond} - A^{(0)} \xv^{(0)}_{\diamond} - A^{(i)} \xv^{(i)}_{\diamond}} = 0$.
    Thus, to upper bound the last term of \eqref{eq:startBoundGap} we subtract from it   $\iprods{\mu_{\diamond}^{(i)}}{\bv^{(i)}_{\diamond} - A^{(0)} \xv^{(0)}_{\diamond} - A^{(i)} \xv^{(i)}_{\diamond}}$ for all $i$ to obtain that it is at most
    \begin{align*}
    \sum_{i \in [s]} \iprod{\mu^{(i)}_{\diamond}}{\bv^{(i)} - A^{(0)} \bar{\xv}^{(0)} - A^{(i)} \bar{\xv}^{(i)}} & = \sum_{i \in [s]} \iprod{\mu^{(i)}_{\diamond}}{\bv^{(i)} - \bv^{(i)}_{\diamond} - A^{(0)} (\bar{\xv}^{(0)} - \xv^{(0)}_{\diamond}) - A^{(i)} (\bar{\xv}^{(i)} - \xv^{(i)}_{\diamond})} \\
&\le \sum_{i \in [s]} \Big\|\mu^{(i)}_{\diamond} \Big\|_1 \cdot \Big\|2d\ones - A^{(0)} (\bar{\xv}^{(0)}-\xv^{(0)}_{\diamond}) - A^{(i)} (\bar{\xv}^{(i)}-\xv^{(i)}_{\diamond})\Big\|_{\infty} \\
&\le \sum_{i \in [s]} \Big\|\mu^{(i)}_{\diamond} \Big\|_1 \cdot \Big\|d\ones - A^{(0)} (\bar{\xv}^{(0)}-\xv^{(0)}_{\diamond})\Big\|_{\infty} \\
&~~~+ \sum_{i \in [s]} \Big\|\mu^{(i)}_{\diamond} \Big\|_1 \cdot \Big\|d\ones - A^{(i)} (\bar{\xv}^{(i)}-\xv^{(i)}_{\diamond})\Big\|_{\infty}.
    \end{align*}
Employing the the previous bound and those of \eqref{eq:gapFirstTerm} and \eqref{eq:gapSecondTerm} to \eqref{eq:startBoundGap} concludes the proof of Lemma~\ref{lemma:startBoundGap}.

\subsection{Bounding the Size of the Optimal Dual (Proof of Lemma \ref{lemma:boundDual})}

\label{sec:bound_dual}

To simplify the notation, let $\xi := \frac{2}{4 \beta_{\min}-1}= \frac{1}{2\alpha} > 1$. We want to prove that with probability at least $1 - \frac{1}{n^2}$, we have $\|\mu^{(i)}_{\diamond}\|_{\infty} \le \xi$ for all blocks $i \in [s] \cup \{0\}$. We first note that it suffices to consider $i \in [s]$, that is, the case $i=0$ comes for free: this is because $\mu^{(0)}_{\diamond}$ is defined as the average $\mu^{(0)}_{\diamond}=\frac{1}{s} \sum_{i = 1}^s \mu^{(i)}_{\diamond}$ and, since norms are convex, by Jensen's inequality we have $\|\mu^{(0)}_{\diamond}\|_{\infty} \le \frac{1}{s} \sum_{i=1}^s \|\mu^{(i)}_{\diamond}\|_{\infty}$.

To prove the desired result, we say that an instance of our problem is \textbf{typical} if every row $r \in [m]$ satisfies the following:
\vspace{-2pt}
    \begin{enumerate}
    \item $ \sum_{j \in [n]} A^{(0)}_{r,j} \leq \frac{n}{2} + \sqrt{n \ln s n m}$;
    \item $\left|\left\{j \in [n]: \cv^{(i)}_j \ge \xi
    \cdot A^{(i)}_{r,j}\right\}\right| \,\le\, \frac{n}{2 \xi} + \sqrt{n \ln s n m}$,  for all blocks $i \in [s]$.
    \end{enumerate}

Notice that, for a fixed row $r \in [m]$, since the entry $A^{(0)}_{r,j}$ is uniformly distributed in $[0,1]$, in expectation we have $\E \sum_{j \in [n]} A^{(0)}_{r,j} = \frac{n}{2}$; moreover, since $\cv^{(i)}_j$ is also uniformly distributed in $[0,1]$, the probability that $\cv^{(i)}_j \ge \xi \cdot A^{(i)}_{r,j}$ can be seen to be $\frac{1}{2\xi}$ (this uses the fact $\xi \ge 1$), and so we expect $\frac{n}{2 \xi}$ of the $j$'s to satisfy this inequality. A quick application of Chernoff's bound show that the extra slack in the definition of a \textbf{typical} block is enough to guarantee that this property holds with high probability for all rows and blocks. 

\begin{claim} \label{claim:typicalBlocks}
    If $sm \geq 2 $, with probability at least $1 - \frac{1}{n^2}$ the instance is \textbf{typical}.   
\end{claim}

\begin{proof}
    Fix a row $r \in [m]$. Since $A^{(0)}_{r,j}$ is uniformly distributed in $[0,1]$, we have $\E \sum_{j \in [n]} A^{(0)}_{r,j} = \frac{n}{2}$, and so Chernoff bound (\Cref{lemma:chernoff} in \Cref{app:prob}) gives $$\Pr\bigg(\sum_{j \in [n]} A^{(0)}_{r,j} > \frac{n}{2} + \sqrt{n \ln s n m}\bigg) \,\le\, e^{-2\ln s n m} = \frac{1}{(snm)^2}.$$ Taking a union bound over all rows, Item 1 in the definition of a typical instance holds with probability at least $1- \frac{1}{(sn)^2 m}$. Also, fix a row $r \in [m]$ and a block $i \in [s]$, and let $X_j$ be the indicator that $\cv^{(i)}_j \ge \xi \cdot A^{(i)}_{r,j}$, so $\sum_{j \in [n]} X_j = |j : \cv^{(i)}_j \ge \xi \cdot A^{(i)}_{r,j}|$. Since $(\cv_j, A^{(i)}_{r,j})$ is uniformly distributed in the square $[0,1]^2$ and $\xi \ge 1$, $\Pr(X_i = 1) = \frac{1}{2 \xi}$ (i.e., the area of the triangle $[0,1]^2 \cap \{(y,x) : y \ge \xi x\}$), and again by Chernoff bound $\Pr(\sum_{j \in [n]} X_j \ge \frac{n}{2 \xi} + \sqrt{n \ln s n m}) \le \frac{1}{(snm)^2}$. Taking a union bound over all $r \in [m]$ and $i \in [s]$, Item 2 in the definition of a typical instance holds with probability at least $1 - \frac{1}{s n^2 m}$. Taking a final union bound over the two items, we conclude the proof of the claim.
\end{proof}

    Then \Cref{lemma:boundDual} follows from the following bound on the optimal dual for typical instances.  

\begin{claim} \label{claim:dualFriendly}
    If the instance is \textbf{typical}, under the assumption that $2 \sqrt{n \ln s n m} + 6\alpha m \log n < n \cdot \big(\beta_{\min} - \frac{1}{4}\big)$ ,then for all blocks $i \in [s]$ we have $\|\mu_\diamond^{(i)}\|_{\infty} \leq \xi$.
\end{claim}

\begin{proof}
   Let $\LAG_{\bv_\diamond}(\mu) := \max\limits_{0 \leq \xv \leq 1} L^{\text{std}}_{\bv_{\diamond}}(\xv,\mu).$ The fact that $\mu_{\diamond}$ is a saddle point for $L^{\text{std}}_{\bv_{\diamond}}$ means that it is an argmin of $\LAG_{\bv_\diamond}(\cdot)$. 

   By means of contradiction, assume that for some block $\bar{i} \in [s]$ we have $\|\mu^{(\bar{i})}\|_{\infty} > \xi$; without loss of generality, assume this is attained at the first coordinate, namely $\mu^{(\bar{i})}_1 > \xi$. First, notice that $\LAG_{\bv_\diamond}(\mu)$ has the the following closed form (we use $(v)^+ : = \max\{0, v\}$ to denote the positive part of a value $v$):
\begin{align}
         \LAG_{\bv_\diamond}(\mu) & = \sum_{j = 1}^n  \Big(s\cdot\cv^{(0)} - \sum_{i=1}^{s} (\mu^{(i)})^{\top} A^{(0)} \Big)^+_j + \sum_{i=1}^{s} \iprod{\mu^{(i)}}{\bv^{(i)}_{\diamond}} + \sum_{i = 1}^s \sum_{j = 1}^n \Big(\cv^{(i)} - (\mu^{(i)})^{\top} A^{(i)} \Big)^+_j, \label{eq:dualClosed}
\end{align}
    which is obtained by rewriting $L^{\text{std}}_{\bv_{\diamond}}$ (defined in equation \eqref{eq:lagStd}) as
    \begin{align*}
    L^{std}_{\bv_{\diamond}}(\xv,\mu) = \iprod{s \cdot \cv^{(0)} - \sum_{i = 1}^s (A^{(0)})^{\top} \mu^{(i)}}{\,\xv^{(0)}} + \sum_{i=1}^{s} \iprod{\mu^{(i)}}{\bv^{(i)}_{\diamond}} + \sum_{i = 1}^s \iprod{\cv^{(i)} - (A^{(i)})^{\top} \mu^{(i)}}{\xv^{(i)}}
    \end{align*}
    and taking $\max$ over $0 \le \xv \le 1$.

    To reach a contradiction, we now construct a solution $\mu_{*}$ with better value $\LAG_{\bv_\diamond}(\mu_*)$ than $\mu_{\diamond}$, contradicting its optimality: define $\mu_*$ to be the same as $\mu_{\diamond}$ except that $\mu^{(\bar{i})}_{*1} = \mu^{(\bar{i})}_{\diamond 1} - \e$, for some positive $\e$. 

    We claim that $\LAG_{\bv_{\diamond}}(\mu_*) < \LAG_{\bv_{\diamond}}(\mu_\diamond)$ for small enough $\e > 0$. The key observation is that the difference these quantities can be bounded as
    \begin{align}
    \LAG_{\bv_{\diamond}}(\mu_*) - \LAG_{\bv_{\diamond}}(\mu_\diamond) \le \sum_{j = 1}^n \e A^{(0)}_{1,j} - \e \bv^{(\bar{i})}_{\diamond 1} + \e \cdot \Big| \Big\{ j \in [n] \,:\, \Big(\cv^{(\bar{i})} - (\mu^{(\bar{i})}_{\diamond})^{\top} A^{(\bar{i})}\Big)_j \ge 0 \Big\} \Big|\,.  \label{eq:diffDual}
    \end{align}    
    which we now justify: Each of the three terms on the RHS come from the three terms on the RHS of \eqref{eq:dualClosed}, with 1) the first coming from the subadditivity of the positive part function $(\cdot)^+$ and the fact $((\mu^{(\bar{i})}_{*})^{\top} A^{(0)})_j = ((\mu^{(\bar{i})}_{\diamond})^{\top} A^{(0)})_j - \e A^{(0)}_{1,j}$; 2) the second term following directly from the definition of $\mu_{*}$; 3) the last coming from the fact that for $i \neq \bar{i}$ we do not pick up any difference between $\mu^{(i)}_{\diamond}$ and $\mu^{(i)}_{*}$ (they are the same by definition) and for $i = \bar{i}$, for small enough $\e$ because of the positive value function {$(\cdot)^+$} we only pick up a difference for $j$'s such that $(\cv^{(\bar{i})} - (\mu^{(\bar{i})}_{\diamond})^{\top} A^{(\bar{i})})_j \ge 0$, in which case we pick up a difference of $\e A^{(\bar{i})}_{1j} \le \e$.

    To further upper bound the RHS of \eqref{eq:diffDual}, we use the assumptions at hand. Since the instance is typical, we directly have $\sum_{j = 1}^n \e A^{(0)}_{1,j} \le \frac{n}{2} + \sqrt{n \ln s n m}$. Also, $-\e \bv^{(\bar{i})}_{\diamond 1} = -\e \bv^{(\bar{i})}_{1} + \e 2 d = - \e 2n \beta^{(\bar{i})}_1 + \e 2 d \le -\e 2n \beta_{\min} + \e 2d$. Finally, the non-negativity of $\mu_{\diamond}$ and $A$ and the assumption $\mu^{(\bar{i})}_{\diamond 1} > \xi$ imply that $(\cv^{(\bar{i})} - (\mu^{(\bar{i})}_{\diamond})^{\top} A^{(\bar{i})})_j \le (\cv^{(\bar{i})} - \mu^{(\bar{i})}_{\diamond 1} \cdot A^{(\bar{i})}_1)_j < (\cv^{(\bar{i})} - \xi \cdot A^{(\bar{i})}_1)_j$, and so Item 2 in the definition of a typical instance implies that $\e \cdot | \{ j \in [n] \,:\, (\cv^{(\bar{i})} - (\mu^{(\bar{i})}_{\diamond})^{\top} A^{(\bar{i})})_j \ge 0 \} | \le \epsilon (\frac{n}{2\xi}  + \sqrt{n \ln s n m}) $.

    Plugging these bounds on \eqref{eq:diffDual}, we get
    \begin{align}
    \LAG_{\bv_{\diamond}}(\mu_*) - \LAG_{\bv_{\diamond}}(\mu_\diamond) &\le \e \cdot \Big(\frac{n}{2} + \frac{n}{2\xi} - 2n\beta_{\min} + 2 \sqrt{n \ln s n m} + 2d\Big) \notag\\
    &= \e \cdot \Big(n \cdot \Big(\frac{1}{4} - \beta_{\min}\Big) + 2 \sqrt{n \ln s n m} + 2d\Big), \label{eq:lastDualBound}
    \end{align}
    where the last inequality uses the definition $\xi = \frac{2}{4 \beta_{\min} - 1}$. By the definition of $d$ in \eqref{eq:defD} and the assumption that
    $2 \sqrt{n \ln s n m} + 6\alpha m \log n < n \cdot \big(\beta_{\min} - \frac{1}{4}\big)$, \eqref{eq:lastDualBound} is strictly negative, which contradicts the optimality of $\mu_{\diamond}$. This concludes the proof of the claim. 
\end{proof}

\subsection{Number of Zeros in the Optimal Primal Solution (Proof of Lemma \ref{lemma:numZeros})} \label{sec:numZeros}

Recall that $\mathcal{N}_0^{(i)}$ denotes the set of coordinates $j$ for which the optimal primal solution $\xv^{(i)}_{\diamond j}$ equals $0$.
In this section, we prove that the probability of the following event $E$:
\vspace{-6pt}
\begin{enumerate}
    \item $\|\mu^{(i)}_{\diamond}\|_{\infty} \le \frac{1}{2\alpha}$ for all blocks $i \in [s] {\cup \{0\}}$, \vspace{-2pt}
    \item but there is a block $i \in [s] {\cup \{0\}}$ such that $\|\mu^{(i)}_{\diamond}\|_1 \ge \frac{f \log n}{2\alpha n}$ and $\big|\mathcal{N}_0^{(i)}\big| < \frac{n}{2^{m+{2}}} \cdot \min\big\{ \frac{1}{4} \normm{\mu^{(i)}_\diamond}_1, 1\big\}$ \vspace{-4pt}
\end{enumerate}
is at most ${(s+1)/n^2}$.

In order to handle the stochastic dependencies between the set $\mathcal{N}_0^{(i)}$ and the optimal dual vector $\mu_{\diamond}$, we will consider each possible value of $\mu_{\diamond}$ separately and then take a union bound over a discretized set of these possibilities to obtain the desired result.

The first observation is that if the column $(\cv^{(i)}_j, \colij)$ has negative reduced cost with respect to the optimal dual $\mu_{\diamond}$ (i.e.,
$\cv^{(i)}_j - \ip{\colij}{\mu^{(i)}_{\diamond}} < 0$), then since $\xv_{\diamond}^{(i)}$ is a saddle point of $L^{std}_{\bv_{\diamond}}$, we must have $\xv^{(i)}_{\diamond j} = 0$, and hence $j \in \mathcal{N}_0^{(i)}$.
Thus, to control the size of $\mathcal{N}_0^{(i)}$, it suffices to control the number of columns with negative reduced cost.
To this end, for a fixed dual vector $\mu^{(i)}$, define the set of columns whose reduced cost with respect to $\mu^{(i)}$ is less than $-\frac{\log n}{\alpha n}$:
\begin{align*}
    \overline{\nrc}^{(i)}_{\mu^{(i)}}
    := \left\{
        j \in [n] :
        \cv^{(i)}_j - \ip{\mu^{(i)}}{\colij}
        < - \frac{\log n}{\alpha n}
    \right\}.
\end{align*}
The additional slack introduced by considering such columns, rather than those with merely negative reduced cost, will be crucial in the sequel.
Next, we construct a discretization of the possible values of $\mu_{\diamond}$.
Recall that an $\varepsilon$-net $\Lambda$ (with respect to the norm $\|\cdot\|_1$) of a set $U \subseteq \mathbb{R}^m$ is a subset $\Lambda \subseteq U$ such that for every $u \in U$, there exists $v \in \Lambda$ satisfying $\|u - v\|_1 \le \varepsilon$.
We construct a $\frac{\log n}{\alpha n}$-net of the set of dual vectors relevant to the event $E$, namely
$\left\{
\mu \in \mathbb{R}^m_+ :
\|\mu\|_{\infty} \le \frac{1}{2\alpha},
\ \|\mu\|_1 \ge \frac{f \log n}{2\alpha n}
\right\}.$
Since the argument is standard, we defer the proof to Appendix~\ref{app:epsNet}.

\begin{claim} \label{claim:epsNet}
    There is a $\frac{\log n}{\alpha n}$-net $\Lambda$ for the set $\{ \mu \in \R^m_+ : \|\mu\|_{\infty} \le \frac{1}{2\alpha}, \|\mu\|_1 \ge \frac{f \log n}{2 \alpha n}\}$, with respect to the norm $\|\cdot\|_1$, of size at most $(\frac{mn}{\log n} + 1)^m \le (2mn)^m$.
\end{claim}

We now formally connect the newly defined sets $\overline{\nrc}^{(i)}_{\mu^{(i)}}$ and $\Lambda$ with the event of interest $E$.

\begin{claim}
    If the event $E$ holds, under the assumption that $f\geq 4$, then there is a block $i \in [s] \cup \{0\}$ and a $\mu^{(i)} \in \overline{\Lambda}$ such that $|\overline{\nrc}^{(i)}_{\mu^{(i)}}| < \frac{n}{2^{m+{2}}} \cdot \min\big\{ \frac{1}{2} \normm{\mu^{(i)}}_1, 1\big\} $.
\end{claim}

\begin{proof}
    Suppose the event $E$ holds, and let $\bar{i}$ be a block such that $\|\mu^{(\bar{i})}_{\diamond}\|_1 \ge \frac{f \log n}{2\alpha n}$ and $\big|\mathcal{N}_0^{(\bar{i})}\big| < \frac{n}{2^{m+{2}}} \cdot \min\big\{ \frac{1}{4} \normm{\mu^{(i)}_\diamond}_1, 1\big\}$ (note in addition we have $\|\mu^{(\bar{i})}_{\diamond}\|_{\infty} \le \frac{1}{2\alpha}$). Let $\mu^{(\bar{i})}$ be a vector in the set $\Lambda$ such that $\|\mu^{(\bar{i})} - \mu^{(\bar{i})}_{\diamond}\|_1 \le \frac{\log n}{\alpha n}$. We claim that $\mu^{(\bar{i})}$ satisfies the desired properties, namely $|\overline{\nrc}^{(\bar{i})}_{\mu^{(\bar{i})}}| < \frac{n}{2^{m+{2}}} \cdot \min\big\{ \frac{1}{2} \normm{\mu^{(\bar{i})}}_1, 1\big\}$.

    To see that, we show that $\overline{\nrc}^{(\bar{i})}_{\mu^{(\bar{i})}} \subseteq \mathcal{N}_0^{(\bar{i})}$: Take $j \in \overline{\nrc}^{(\bar{i})}_{\mu^{(\bar{i})}}$; this implies
    \begin{align*}
    \cv^{(\bar{i})}_j + \frac{\log n}{\alpha n} < \ip{\av_{j}^{(\bar{i})}}{\mu^{(\bar{i})}} = \ip{\av_{j}^{(\bar{i})}}{\mu^{(\bar{i})}_{\diamond}} + \ip{\av_{j}^{(\bar{i})}}{\mu^{(\bar{i})} - \mu^{(\bar{i})}_{\diamond}} \le  \ip{\av_{j}^{(\bar{i})}}{\mu^{(\bar{i})}_{\diamond}}  +\frac{\log n}{\alpha n}, 
    \end{align*}
    the last inequality using the fact $\av_{j}^{(\bar{i})} \in [0,1]^m$; reorganizing the terms this implies that $\cv^{(\bar{i})}_j - \ip{\av_{j}^{(\bar{i})}}{\mu^{(\bar{i})}_\diamond} < 0$, and so $j \in \mathcal{N}_0^{(\bar{i})}$, giving the desired result. 

    Next, we show that $\frac{1}{2} \normm{\mu^{(\bar{i})}_\diamond}_1 \le \normm{\mu^{(\bar{i})}}_1$: using triangle inequality
    \begin{align*}
        \normm{\mu^{(\bar{i})}_\diamond}_1 = \normm{\mu^{(\bar{i})} + (\mu^{(\bar{i})}_\diamond - \mu^{(\bar{i})})}_1 \le \normm{\mu^{(\bar{i})}}_1 + \normm{\mu^{(\bar{i})}_\diamond - \mu^{(\bar{i})}}_1 \le \normm{\mu^{(\bar{i})}}_1 + \frac{\log n}{\alpha n} \le \normm{\mu^{(\bar{i})}}_1 + \frac{1}{2}\|\mu^{(\bar{i})}_{\diamond}\|_1,
    \end{align*}
    where the last inequality uses the fact $\|\mu^{(\bar{i})}_{\diamond}\|_1 \ge \frac{f \log n}{2\alpha n} \ge \frac{2 \log n}{\alpha n}$; rearranging gives the desired result. 

    Combining the previous two paragraphs we then get
    \begin{align*}
      \big|\overline{\nrc}^{(\bar{i})}_{\mu^{(\bar{i})}}\big| \,\le\, \big|\mathcal{N}_0^{(\bar{i})}\big| \,\le\,  \frac{n}{2^{m+{2}}} \cdot \min\Big\{ \frac{1}{4} \normm{\mu^{(i)}_\diamond}_1, 1\Big\}\,\le\,  \frac{n}{2^{m+{2}}} \cdot \min\Big\{ \frac{1}{2} \normm{\mu^{(\bar{i})}}_1, 1\Big\},
    \end{align*}
    thus proving the claim. 
\end{proof}

We now bound the probability that one of these sets $\overline{\nrc}^{(i)}_{\mu^{(i)}}$ is small, which together with the previous result upper bounds the probability of the event $E$ and gives the proof of Lemma  \ref{lemma:numZeros}.

\begin{claim}
Under the assumption that $f \geq \max\left\{8,\frac{\alpha 2^{m+6} \log (2mn)}{\log n}\right\}$
it holds
\begin{align*}
    \Pr\bigg(\textrm{there is $i \in [s] {\cup \{0\}}$ and $\mu^{(i)} \in \Lambda$ such that $|\overline{\nrc}^{(i)}_{\mu^{(i)}}| < \frac{n}{2^{m+{2}}} \cdot \min\bigg\{ \frac{1}{{2}} \normm{\mu^{(i)}}_1, 1\bigg\}$}  \bigg) \le \frac{s+1}{n^{2}}.
\end{align*}
\end{claim}

\begin{proof}    
    For each $j$, notice that whenever $\colij \in \big[\frac{1}{2}, 1\big]^m$ and $\cv_j^{(i)} \in \big[0,\frac{1}{4} \normm{\mu^{(i)}}_1\big)$, we have
    \begin{align*}
        \cv^{(i)}_j - \ip{\colij}{\mu^{(i)}} \,\le\, \cv^{(i)}_j - \frac{1}{2} \normm{\mu^{(i)}}_1 \,<\, -\frac{1}{4} \normm{\mu^{(i)}}_1 \,\le\, -\frac{\log n}{\alpha n},
    \end{align*}
where the last inequality uses the fact $\normm{\mu^{(i)}}_1 \ge \frac{f \log n}{2 \alpha n} \ge \frac{4 \log n}{\alpha n}$ since $\mu^{(i)} \in \Lambda$. Thus, when this happens, the index $j$ belongs to the set $\overline{\nrc}^{(i)}_{\mu^{(i)}}$. Since $\cv_j^{(i)}$ and $\colij$ have entries uniformly distributed in $[0,1]$, this observation implies that $j$ belongs to $\overline{\nrc}^{(i)}_{\mu^{(i)}}$ with probability at least $(\frac{1}{2})^m \cdot \min\{ \frac{1}{{4}} \normm{\mu^{(i)}}_1, 1\}$. 

Therefore, the expected size of this set, $\E\, |\overline{\nrc}^{(i)}_{\mu^{(i)}}|$, is at least $n \cdot (\frac{1}{2})^m \cdot \min\{ \frac{1}{{4}} \normm{\mu^{(i)}}_1, 1\}$, and further applying the multiplicative Chernoff bound (Lemma \ref{lemma:chernoff}) over the indicator that each $j \in \overline{\nrc}^{(i)}_{\mu^{(i)}}$ we get
%
\begin{align*}
    \Pr\bigg( |\overline{\nrc}^{(i)}_{\mu^{(i)}}| < \frac{n}{2^{m+1}} \cdot \min\bigg\{ \frac{1}{{4}} \normm{\mu^{(i)}}_1, 1\bigg\} \bigg) \le \exp\bigg(- \frac{n}{2^{m+4}} \cdot \min\bigg\{  \frac{1}{{2}} \normm{\mu^{(i)}}_1, 1\bigg\}\bigg) \le \frac{1}{(2mn)^{m+2}},
\end{align*}
where the last inequality uses again that $\|\mu^{(i)}\|_1 \ge \frac{f\log n}{2 \alpha n}\ge \frac{2^{m+{5}} (m+{2}) \log (2mn)}{n}$ since $\mu^{(i)} \in \Lambda$ and the assumption of $f$.

This implies that 
\begin{align*}
    \Pr\bigg( |\overline{\nrc}^{(i)}_{\mu^{(i)}}| < \frac{n}{2^{m+{2}}} \cdot \min\bigg\{ \frac{1}{{2}} \normm{\mu^{(i)}}_1, 1\bigg\} \bigg) & \leq \Pr\bigg( |\overline{\nrc}^{(i)}_{\mu^{(i)}}| < \frac{n}{2^{m+1}} \cdot \min\bigg\{ \frac{1}{{4}} \normm{\mu^{(i)}}_1, 1\bigg\} \bigg) \leq \frac{1}{(2mn)^{m+{2}}}
\end{align*}

Finally, taking a union bound over all blocks $i \in [s]$ and all vectors $\mu^{(i)} \in \Lambda$, we obtain 
\begin{align*}
    &\Pr\bigg(\textrm{there is $i \in [s] {\cup \{0\}} $ and $\mu^{(i)} \in \Lambda$ such that $|\overline{\nrc}^{(i)}_{\mu^{(i)}}| < \frac{n}{2^{m+{2}}} \cdot \min\bigg\{ \frac{1}{{2}} \normm{\mu^{(i)}}_1, 1\bigg\}$}  \bigg) \\
    &\le     \sum_{i \in [s] {\cup \{0\}}} \sum_{\mu^{(i)} \in \Lambda} \Pr\bigg(|\overline{\nrc}^{(i)}_{\mu^{(i)}}| < \frac{n}{2^{m+{2}}} \cdot \min\bigg\{ \frac{1}{{2}} \normm{\mu^{(i)}}_1, 1\bigg\} \bigg)  \le ({s+1}) \cdot |\Lambda| \cdot \frac{1}{(2mn)^{m+{2}}} \le \frac{s+1}{n^2}.
\end{align*}
This concludes the proof of the claim. 
\end{proof}


\subsection{Discrepancy Lemma} \label{sec:prelim}


We next introduce a discrepancy lemma due to Dyer and Frieze, which is central
to the proof of the integrality gap bound. We present a slightly
simplified version of the result that is cleaner and sufficient for our
purposes.
We also note that this discrepancy lemma may be further improved under
alternative assumptions on the stochastic model using more sophisticated
techniques; for related developments, we refer the reader
to~\cite{borst2023integrality,borst2023integrality_discrepancy}.
\begin{lem}
    \cite{dyer1989probabilistic}
    \label{lem_discrepancy}
    Fix constants $a, M, \e > 0$. Then for $k$ being at least a sufficiently large integer (depending on $a, \alpha, M$) the following holds.
    
Let $Y^1, \ldots, Y^{2k}$ be $2k$ random vectors in $[-a,a]^m$ where all the coordinates, across all vectors, are independent and identically distributed. Further assume that each coordinate $Y^j_i$ satisfies $\E Y^j_i = 0$ and $\mathrm{Var}(Y^j_i) = 1$ and has a continuous density $g$ satisfying $g(x) \leq M$. For every deterministic vector $\vv \in \R^m$ satisfying $\|\vv\|_{\infty} \le k^{1/2 - \e}$, we have 
\begin{align*}
\P\left( \exists K \subseteq \{1, 2, \ldots, 2k\} : |K| = k,\ \bigg\|\sum_{j \in K} Y^j - \vv\bigg\|_{\infty} \le \frac{3k}{4^{k/m}} \right) 
\,\geq\, \frac{1}{2^{m+2}}.
\end{align*}
\end{lem}

\subsection{Existence of Good Rounding (Proof of Lemma \ref{lemma:condGap})} \label{sec:goodRounding}

Recall the definition of the event $Good$: it occurs whenever both of the
following conditions hold:
\begin{enumerate}
    \item For every block $i \in [s] \cup \{0\}$, we have
    $\big\|\mu_{\diamond}^{(i)}\big\|_{\infty} \leq \frac{1}{2\alpha}$.
    \item For every block $i \in [s] \cup \{0\}$, if
    $\big\|\mu_{\diamond}^{(i)}\big\|_1 \ge \frac{f \cdot \log n}{2 \alpha n}$,
    then
    $\big|\mathcal{N}_0^{(i)}\big|
        \geq \frac{n}{2^{m+2}} \cdot
        \min\Big\{ \tfrac{1}{4}\normm{\mu^{(i)}_\diamond}_1, 1\Big\}.$
\end{enumerate}

Our goal is to prove that, conditioned on the event $Good$, with probability at
least $1 - \frac{s+1}{n^2} - (s+1)\pi^f$ we obtain a small integrality gap,
namely,$
\Delta \le O_m\!\left(\frac{s f \log^2 n}{n}\right).$
Recall that Lemma~\ref{lemma:startBoundGap} states that for any feasible solution
$\bar{\xv}$ to the original (integer) problem~\eqref{eq:IP} that preserves all
$1$-entries of $\xv^{(i)}_{\diamond}$ (i.e., if
$\xv^{(i)}_{\diamond j} = 1$ then $\bar{\xv}^{(i)}_j = 1$), the integrality gap
$\Delta$ can be bounded by
\begin{align}
    \Delta \leq s \cdot \Delta^{(0)}(\bar{\xv}^{(0)}) + \sum_{i=1}^{s} \Delta^{(i)}(\bar{\xv}^{(i)}),
    \label{eq:deltaRecap}
\end{align}
where
\begin{align}
\Delta^{(i)}(\bar{\xv}^{(i)}) =
\big\|\mu^{(i)}_{\diamond}\big\|_1 \cdot
\big\| A^{(i)} (\xv^{(i)}_{\diamond} - \bar{\xv}^{(i)}) + d \ones\big\|_{\infty}
\;+\;
\sum_{j : \bar{\xv}^{(i)}_j > \xv^{(i)}_{\diamond j}}
\Big(\langle \mu^{(i)}_{\diamond}, \av^{(i)}_j\rangle - c^{(i)}_j\Big).
\label{eq:delta_i}
\end{align}

Thus, it suffices to show that under $Good$ there exists a rounding $\bar{\xv}$
such that all the quantities $\Delta^{(i)}(\bar{\xv}^{(i)})$ are small.
To this end, we further refine the conditioning on $Good$.
Specifically, we will further condition on (i) the set of coordinates where the optimal
primal Lagrangian solution $\xv_{\diamond}$ equals $0$, and (ii) the exact data
(i.e., $\cv^{(i)}_j$ and $A^{(i),j}$) on the remaining coordinates.
As we will see later (Lemma~\ref{lem:cond_indp}), this conditioning fixes the
optimal dual vector $\mu_{\diamond}$, but leaves the columns on the nonzero
coordinates independent and uniformly distributed, subject only to having
nonpositive reduced cost.
This leaves enough randomness for us to find the desired rounding with high
probability by using some of these columns.
To formalize this, let
$
\cN_0 := \{(i,j) : \xv^{(i)}_{\diamond j} = 0\}$
be the set of coordinates where the optimal primal Lagrangian solution equals
$0$.

\begin{lemma} \label{lemma:mainDeltai}
Consider a conditioning of the set $\mathcal{N}_0$ and of the column data
$((\cv^{(i)}_j, \colij))_{(i,j) \notin \cN_0}$ in the complement of $\cN_0$ that belongs to the event $Good$. Under the condition of Theorem \ref{thm:smallGap},
then with (conditional) probability at least $1 - {\frac{(s+1)}{n^2}-(s+1) \pi^f}$ there exists an integer solution $\bar{\xv} \in \{0,1\}^{n \cdot (s+1)}$ that:\footnote{Actually because of Lemma~\ref{lem:cond_indp}, this result does not hold for all the conditionings of the set $\mathcal{N}_0$ and $((\cv^{(i)}_j, A^{(i)}_j))_{(i,j) \notin \cN_0}$, there is a set of them of measure 0 where the result may not hold. Since such set of measure 0 will not affect the use of this lemma, we chose to not include this subtlety in the statement.} 
\vspace{-8pt}
\begin{enumerate}
    \item Keeps all $1$-entries of $\xv^{(i)}_{\diamond}$ ;\vspace{-2pt}
    \item Is feasible for \eqref{eq:IP} \vspace{-4pt};
    \item Has $\Delta^{(i)}(\bar{\xv}^{(i)}) \leq O_m(\frac{{f}\log^2 n}{n})$ for all $i \in [s] \cup \{0\}$;
\end{enumerate} 
{where $\pi :=  \big(1 - \frac{1}{2^{m+2}}\big)^{\frac{\alpha^m}{2^{m+6}m}} \in (0,1)$.}
\end{lemma}

This result directly implies the desired bound on the integrality gap, i.e., Lemma \ref{lemma:condGap}.

\begin{proof}[~of Lemma \ref{lemma:condGap}]
    Consider a conditioning of $\mathcal{N}_0$ and $((\cv^{(i)}_j, \colij)_{(i,j) \notin \cN_0}$ that belongs to the event $Good$; by employing the bound from Lemma \ref{lemma:mainDeltai} on \eqref{eq:deltaRecap}, under this conditioning we have $$\Delta \le  O_m\bigg(\frac{{sf}\log^2 n}{n}\bigg)$$ with probability at least $1 {- \frac{s+1}{n^2}-(s+1)\pi^f}
    $. Taking expectation over all the conditionings of $\mathcal{N}_0$ and $((\cv^{(i)}_j, \colij)_{(i,j) \notin \cN_0}$ that belongs to the event $Good$ gives
    \begin{align*}
        \Pr\bigg(\Delta \le O_m\bigg(\frac{{sf}\log^2 n}{n}\bigg) ~\bigg|~ Good \bigg) \ge 1 {- \frac{s+1}{n^2}-(s+1)\pi^f}
    ,
    \end{align*}
    as desired.
\end{proof}

    So for the remainder of the section we prove Lemma \ref{lemma:mainDeltai}. 


\subsubsection{Proof of Lemma~\ref{lemma:mainDeltai}}

We start by formalizing the previous claim that conditioning on the index set
$\cN_0$ and on the columns $\big((\cv^{(i)}_j, A^{(i)}_j)\big)_{(i,j)\notin \cN_0}$
at the remaining indices fixes the optimal dual vector $\mu_{\diamond}$ (almost
surely), while leaving the columns on the zero coordinates independent and
uniformly distributed, subject only to having nonpositive reduced cost.
To this end, for a dual vector $\mu^{(i)} \in \R_+^m$, define
\begin{align*}
    \nrc^{(i)}_{\mu^{(i)}} :=
    \{ (z,\yv) \in [0,1]^{1+m} : z - \ip{\yv}{\mu^{(i)}} \le 0 \}.
\end{align*}
Thinking of $(z,\yv)$ as the data of a column in our problem (where $\zv$ represents
the gain), $\nrc^{(i)}_{\mu^{(i)}}$ is the set of column values that yield
nonpositive reduced cost with respect to the dual vector $\mu^{(i)}$.
We will use the following result from
\cite{dyer1989probabilistic,borst2023integrality}.
Although the result is stated for a single-block instance, it applies directly
here by viewing \eqref{eq:IP} as a single-block instance.

\begin{lemma}\cite{dyer1989probabilistic,borst2023integrality}
    \label{lem:cond_indp}
    Conditional on the value of the set $\mathcal{N}_0$, the data $((\cv^{(i)}_j, \colij))_{(i,j) \notin \cN_0}$ in the remaining indices uniquely determines the optimal primal/dual Lagrangian solutions $\xv_{\diamond},\mu_{\diamond}$ almost surely. If we further condition on an exact value of $((\cv^{(i)}_j, \colij))_{(i,j) \notin \cN_0}$ that determines $\xv_{\diamond},\mu_{\diamond}$ uniquely, then the columns $(\cv^{(i)}_j,\colij)$ for $(i,j) \in \cN_0$ are independent and each is uniformly distributed in the set $\nrc^{(i)}_{\mu^{(i)}_{\diamond}}$.
\end{lemma}

For the remainder of the proof, we condition on the value of the index set
$\cN_0$ and on the data $\big((\cv^{(i)}_j, \colij)\big)_{(i,j)\notin \cN_0}$ for
the remaining columns, such that the pair $(\xv_{\diamond}, \mu_{\diamond})$ is
uniquely defined and the event $Good$ holds.
That is, for every block $i \in [s] \cup \{0\}$ we have
$\big\|\mu_{\diamond}^{(i)}\big\|_{\infty} \leq \frac{1}{2\alpha}$, and whenever
$\big\|\mu_{\diamond}^{(i)}\big\|_1 \ge \frac{f \cdot \log n}{2 \alpha n}$ we also
have
$\big|\mathcal{N}_0^{(i)}\big|
\geq
\frac{n}{2^{m+2}} \cdot
\min\Big\{ \tfrac{1}{4} \normm{\mu^{(i)}_\diamond}_1, 1 \Big\}.$
We denote this conditioning (i.e., the set of outcomes satisfying these
conditions) by $Cond$.

We now construct the desired rounded solution
$\bar{\xv} = (\bar{\xv}^{(0)}, \ldots, \bar{\xv}^{(s)}) \in \{0,1\}^{(s+1)\cdot n}$
by the following two steps:
(i) rounding down the fractional coordinates of $\xv_{\diamond}$, and
(ii) setting some of the originally zero coordinates to $1$.
More precisely, let $\lfloor \xv^{(i)}_{\diamond} \rfloor$ denote the vector
obtained by rounding down $\xv^{(i)}_{\diamond}$ componentwise, and let
$\chi_S \in \{0,1\}^n$ denote the indicator vector of a set $S \subseteq [n]$.
We define $
\bar{\xv}^{(i)} := \lfloor \xv^{(i)}_{\diamond} \rfloor + \chi_{T^{(i)}},$
for some set $T^{(i)} \subseteq \cN_0^{(i)}$ of coordinates that are set to $1$.

By construction, the vector $\bar{\xv}^{(i)}$ preserves all $1$-entries of
$\xv^{(i)}_{\diamond}$, which establishes Item~1 of
Lemma~\ref{lemma:mainDeltai}.
We now construct the sets $T^{(i)}$, for each block
$i \in [s] \cup \{0\}$, so as to satisfy
\begin{gather}
    A^{(i)} \bar{\xv}^{(i)}
    \le A^{(i)} \xv^{(i)}_{\diamond} + d \ones,
    \label{eq:mainDeltaItem2} \\[2pt]
    \Delta^{(i)}(\bar{\xv}^{(i)})
    \le O_m\!\left(\frac{f \log^2 n}{n}\right).
    \label{eq:mainDeltaItem3}
\end{gather}

The second inequality is precisely Item~3 of
Lemma~\ref{lemma:mainDeltai}.
The first inequality implies Item~2, namely the feasibility of $\bar{\xv}$.
Indeed, recall that the original right-hand side for scenario $i$ in the packing
problem is
$\bv^{(i)} = \bv^{(i)}_{\diamond} + 2d\ones$.
Since $\xv_{\diamond}$ is feasible for the modified right-hand side, we have
\[
A^{(0)} \xv_{\diamond}^{(0)} + A^{(i)} \xv_{\diamond}^{(i)}
\le \bv^{(i)}_{\diamond}.
\]
Therefore, inequality~\eqref{eq:mainDeltaItem2} implies
\[
A^{(0)} \bar{\xv}^{(0)} + A^{(i)} \bar{\xv}^{(i)}
\le \bv^{(i)}_{\diamond} + 2d\ones
= \bv^{(i)},
\]
which establishes the feasibility of $\bar{\xv}$.

Now for each block $i \in [s] \cup \{0\}$, we define the set $T^{(i)}$ to satisfy \eqref{eq:mainDeltaItem2} and \eqref{eq:mainDeltaItem3} (with high probability), breaking into two cases based on the size of $\normm{\mu^{(i)}_{\diamond}}_1$ as follows.

\paragraph{Case 1 if $\normm{\mu_{\diamond}^{(i)}}_1 < \frac{f \log n}{2\alpha n}$:} 
Here we choose the set $T^{(i)}$ to be empty, that is, no additional coordinate is
set to $1$ in $\bar{\xv}^{(i)}$, so that
$\bar{\xv}^{(i)} = \lfloor \xv^{(i)}_{\diamond} \rfloor$.
Clearly, $\bar{\xv}^{(i)}$ satisfies~\eqref{eq:mainDeltaItem2}.
Moreover, in this case $\Delta^{(i)}(\bar{\xv}^{(i)})$ (recall its definition
in~\eqref{eq:delta_i}) reduces to
\begin{align*}
\Delta^{(i)}(\bar{\xv}^{(i)})
&=
\big\|\mu^{(i)}_{\diamond}\big\|_1 \cdot
\big\| A^{(i)} \big(\xv^{(i)}_{\diamond}
- \lfloor \xv^{(i)}_{\diamond} \rfloor \big) + d \ones \big\|_{\infty} \\
&<
\frac{f \log n}{2\alpha n} \cdot
\big\| A^{(i)} \big(\xv^{(i)}_{\diamond}
- \lfloor \xv^{(i)}_{\diamond} \rfloor \big) + d \ones \big\|_{\infty},
\end{align*}
where the inequality follows from the assumption
$\|\mu^{(i)}_{\diamond}\|_1 < \frac{f \log n}{2\alpha n}$.

To upper bound the remaining term, we first observe that the optimal solution
$\xv_{\diamond}$ is a \emph{basic} optimal solution (by uniqueness).
By standard arguments, $\xv^{(i)}_{\diamond}$ therefore has at most $m$
fractional coordinates; consequently,
$\xv^{(i)}_{\diamond} - \lfloor \xv^{(i)}_{\diamond} \rfloor$
has at most $m$ nonzero entries (see Appendix~\ref{app:fractional} for details).
Since all entries of $A^{(i)}$ are at most $1$, this implies
\begin{align}
0 \le
A^{(i)} \big(\xv^{(i)}_{\diamond}
- \lfloor \xv^{(i)}_{\diamond} \rfloor \big)
\le m \ones.
\label{eq:fractional}
\end{align}

Substituting~\eqref{eq:fractional} into the previous bound yields
\begin{align*}
\Delta^{(i)}(\bar{\xv}^{(i)})
&<
\frac{f \log n}{2\alpha n} \cdot (m + d)
\;\le\;
\frac{f \log n}{2\alpha n} \cdot 4 m \log n
\;\le\;
O_m\!\left(\frac{f \log^2 n}{n}\right),
\end{align*}
where the second inequality uses the definition of $d$ in~\eqref{eq:defD},
which implies $d \le 3\alpha m \log n \le 3 m \log n$.
Thus, $\bar{\xv}^{(i)}$ satisfies~\eqref{eq:mainDeltaItem3}, completing the
analysis of this case.

\paragraph{Case 2 if $\normm{\mu_{\diamond}^{(i)}}_1 \ge \frac{f  \log n}{2\alpha n}$:} In this case the construction of the set $T^{(i)}$ is significantly more involved. 

From Lemma~\ref{lem:cond_indp}, we know that, under the conditioning $Cond$, the
columns indexed by $j \in \cN_0^{(i)}$ of block $i$ remain independent and
uniformly distributed over the set
$\nrc^{(i)}_{\mu^{(i)}_\diamond}$ of column values with nonpositive reduced cost
with respect to $\mu^{(i)}_\diamond$.
However, the nonpositive reduced cost constraint still induces undesirable
correlations among the entries of columns in $\cN_0^{(i)}$.
To obtain additional structure, we therefore restrict attention to a box-like
subset of $\nrc^{(i)}_{\mu^{(i)}_\diamond}$ that consists only of column values
whose reduced costs are not too negative, so that their contribution to
$\Delta^{(i)}$ remains controlled.

\begin{claim}
\label{claim:good_col}
For every $i \in [s] \cup \{0\}$ and vector $\mu^{(i)} \in \R^m_+$ such that $\frac{f \log n}{2\alpha n} \le \norm{\mu^{(i)}}_1$ and $\norm{\mu^{(i)}}_\infty \le \frac{1}{2\alpha}$, there exists a subset $M_{\mu^{(i)}}^{(i)} \subseteq \nrc_{\mu^{(i)}}^{(i)}$ with following properties:
\begin{enumerate}
    \item \vspace{-4pt} (Small reduced cost wrt $\mu^{(i)}$) Every vector $(z,\yv) \in M^{(i)}_{\mu^{(i)}}$ satisfies $z  - \ip{\yv}{\mu^{(i)}} \ge -\frac{f \log n}{n}$.


    \item \vspace{-4pt}  (Box-like) If $(Z,Y) \in \R \times \R^m$ is a random variable uniformly distributed in $M_{\mu^{(i)}}^{(i)}$, then $Y$ is uniformly distributed in $[\alpha,2\alpha]^m$.

    \item \vspace{-4pt} 
    (Relative volume) $\dfrac{\vol\big(M^{(i)}_{\mu^{(i)}}\big)}{\vol\big(\nrc^{(i)}_{\mu^{(i)}}\big)} \ge \dfrac{f \log n \cdot \alpha^{m}}{n \cdot \min\big\{\normm{\mu^{(i)}_\diamond}_1, 1 \big\}}$. 
\end{enumerate}
\end{claim}

\begin{proof}[~sketch]
The region $M^{(i)}_{\mu^{(i)}}$ is defined as 
%
    \begin{align*}
    M^{(i)}_{\mu^{(i)}} := \bigcup_{y \in [\alpha,2\alpha]^m} \bigg(\Big[\ip{y}{\mu^{(i)}} - \tfrac{f\log n}{n}\,,\, \ip{y}{\mu^{(i)}}\Big] \times \{y\}\bigg).
    \end{align*} 
That is, we place the interval $[\ip{y}{\mu^{(i)}} - \tfrac{f\log n}{n}, \ip{y}{\mu^{(i)}}]$ on the first coordinate for each completion $y \in [\alpha, 2\alpha]^m$ of the other coordinates. It is easy to verify that this region satisfies the desired properties, for details see Appendix \ref{app:claimGoodCol}.
\end{proof}

Since, conditioned on $Cond$, each column $(\cv^{(i)}_j, \colij)$ with
$j \in \cN^{(i)}_0$ is uniformly distributed over
$\nrc^{(i)}_{\mu^{(i)}_{\diamond}}$, Item~3 of the previous claim implies that
such columns satisfy
\[
\Pr\Big( (\cv^{(i)}_j, \colij) \in M^{(i)}_{\mu^{(i)}_{\diamond}}
\;\Big|\; Cond \Big)
\;\ge\;
\frac{f \log n \cdot \alpha^{m}}{n \cdot \min\{\normm{\mu^{(i)}_\diamond}_1, 1\}}.
\]
By the definition of $Cond$, the set $\cN^{(i)}_0$ contains at least
$\frac{n}{2^{m+2}} \cdot
\min\Big\{ \tfrac{1}{4}\normm{\mu^{(i)}_\diamond}_1, 1 \Big\}$
elements. Therefore, the expected number of columns in $\cN^{(i)}_0$ that fall
into $M^{(i)}_{\mu^{(i)}_{\diamond}}$ is at least
\begin{align*}
\frac{f \log n \cdot \alpha^{m}}{n \cdot \min\{\normm{\mu^{(i)}_\diamond}_1, 1\}}
\cdot
\frac{n}{2^{m+2}}
\cdot
\min\Big\{ \tfrac{1}{4}\normm{\mu^{(i)}_\diamond}_1, 1 \Big\}
\;\ge\;
\frac{f \log n \cdot \alpha^{m}}{2^{m+4}}.
\end{align*}

Moreover, using the conditional independence of these columns, we may apply the
Chernoff bound (Lemma~\ref{lemma:chernoff}) conditioned on $Cond$ to obtain a
high-probability version of this statement.

\begin{claim} \label{claim:chernoffM}
    Let $\cM^{(i)} := \Big\{j \in \cN^{(i)}_0 : (\cv^{(i)}_j, \colij) \in M^{(i)}_{\mu^{(i)}_{\diamond}}\Big\}$ denote the set of columns of $\cN^{(i)}_0$ landing in $M^{(i)}_{\mu^{(i)}_{\diamond}}$. Then
    \vspace{-8pt}
    \begin{align*}
        \Pr\bigg(|\cM^{(i)}| \ge \frac{f \log n \cdot \alpha^m}{2^{m+{5}}} ~\bigg|~ Cond \bigg) \,\ge\, 1 - \exp\bigg(- \frac{f \log n \cdot \alpha^m}{2^{m+{7}}} \bigg) \ge 1 - \frac{1}{n^2},
    \end{align*}   
    where the last inequality follows from the assumption of $f \geq \frac{2^{2m+8}}{\alpha^m}$.
\end{claim}

We now further condition on the set of ``nice'' columns
$\cM^{(i)}$ being equal to a fixed set $S \subseteq [n]$ that $|S| \geq \frac{f \log n \cdot \alpha^m}{2^{m+{5}}} $.
That is, we define $Cond_S \subseteq Cond$ as the subset of scenarios in $Cond$
for which $\cM^{(i)} = S$.
Note that, conditioned on $Cond$, the columns indexed by $\cN^{(i)}_0$ are
independent and uniformly distributed over
$\nrc^{(i)}_{\mu^{(i)}_\diamond}$, and since
$M^{(i)}_{\mu^{(i)}_{\diamond}} \subseteq \nrc^{(i)}_{\mu^{(i)}_\diamond}$,
conditioning further on $Cond_S$ implies that the columns with indices in $S$
are independent and uniformly distributed over
$M^{(i)}_{\mu^{(i)}_{\diamond}}$.
By Item~2 of Claim~\ref{claim:good_col}, this implies that for each $j \in S$,
the vector $\colij$ is uniformly distributed over $[\alpha, 2\alpha]^m$.

Fix such a set $S$ with cardinality at least
$\frac{f \log n \cdot \alpha^m}{2^{m+5}}$, and partition it into
$f' := \frac{f \alpha^m}{2^{m+6} m}$
disjoint subsets $S_1,\ldots,S_{f'}$, each of size exactly $2m \log n$.
(The definition of $f$ guarantees that $S$ is sufficiently large for such a
partition.) 
We will apply the Discrepancy Lemma (Lemma~\ref{lem_discrepancy}) to each subset
$S_t$ in order to identify a suitable set
$T^{(i)} \subseteq S$, which will serve as the additional coordinates of
$\bar{\xv}^{(i)}$ that are set to $1$.
Although this construction succeeds for each $S_t$ only with moderate
probability, aggregating over all subsets $S_t$ amplifies the probability of
finding an appropriate set of coordinates.
Finally, define
$
\zv^{(i)} :=
A^{(i)} \big(\xv^{(i)}_{\diamond}
- \lfloor \xv^{(i)}_{\diamond} \rfloor \big) + d \ones.$

\begin{claim} \label{claim:useDiscrepancy}
    For each set $S_t$, we have the following:
    %
    \begin{align*}
        \Pr\bigg( \textrm{there is } T^{(i)}_t \subseteq S_t \textrm{ such that } \sum_{j \in T^{(i)}_t} A^{(i),j} \in \bigg[\zv^{(i)} - \frac{\ones}{n},~\zv^{(i)} \bigg] ~\bigg|~ Cond_S\bigg) \ge  \frac{1}{2^{m+2}}.
    \end{align*}
\end{claim}

\begin{proof}
    Recall that for each $j \in S_t$, conditioned on $Cond_S$ the vector $A^{(i),j}$ is uniformly distributed in $[\alpha, 2\alpha]^m$, and independent across $j$'s; thus, if we renormalize it as
\begin{align*}
    Y^j := \dfrac{2\sqrt{3}}{\alpha} \bigg(A^{(i),j} - \dfrac{3\alpha }{2} \ones\bigg),
\end{align*}
we see that each coordinate of this vector has mean 0 and variance 1, and the $Y^j_r$'s are independent across $j$ and $r$. Then define $k := \frac{|S_t|}{2} = m \log n$, $\theta := \frac{\alpha}{2 \sqrt{3}} \cdot \frac{3k}{4^{k/m}}$, and the vector $$\vv := \frac{2 \sqrt{3}}{\alpha} \cdot \bigg(\zv^{(i)} - \bigg(\theta + \frac{3\alpha k}{2}\bigg) \ones\bigg).$$ 
Now we can apply the Discrepancy Lemma (Lemma \ref{lem_discrepancy}) to the random vectors $Y^j$ and the target vector $\vv$ to obtain that with probability at least $\frac{1}{2^{m+2}}$ there is a set $T^{(i)}_t \subseteq S$ of size $k$ such that $\Big\|\sum_{j \in T^{(i)}_t} Y^j - \vv\Big\|_{\infty} \le \frac{3k}{4^{k/m}}$. Unraveling the definitions, this is equivalent to $\sum_{j \in T^{(i)}_t} A^{(i),j} \in [\zv^{(i)} - 2\theta \ones, \zv^{(i)}]$. To conclude the proof, it suffices to show that $2\theta \le \frac{1}{n}$:
\begin{align*}
    2\theta = \frac{\alpha}{\sqrt{3}} \cdot \frac{3 m \log n}{4^{\log n}} \le \frac{m \log n}{n^2} \le \frac{1}{n},
\end{align*}
where the last inequality uses the assumption on the size of $n$.

Actually, in order to apply this discrepancy lemma, we crucially need to satisfy the requirement $\|\vv\|_\infty \le k^{1/2-\e}$ for some $\e > 0$. Recalling that $\zv^{(i)} := A^{(i)} \big(\xv^{(i)}_{\diamond} - \lfloor \xv^{(i)}_\diamond \rfloor) + d \ones$, the definition of $\vv$, and noticing that we defined $d = \theta + \frac{3\alpha k}{2}$ in \eqref{eq:defD}, we have
\begin{align*}
    \|\vv\|_{\infty} \,\le\, \frac{2\sqrt{3}}{\alpha} \cdot \Big\|A^{(i)} \big(\xv^{(i)}_{\diamond} - \lfloor \xv^{(i)}_\diamond \rfloor) + \Big(d-\theta - \frac{3\alpha k}{2}\Big)\ones  \Big\|_{\infty} \,=\, \frac{2\sqrt{3}}{\alpha} \cdot \Big\|A^{(i)} \big(\xv^{(i)}_{\diamond} - \lfloor \xv^{(i)}_\diamond \rfloor) \Big\|_{\infty} \le \frac{2 \sqrt{3} m}{\alpha},
\end{align*}
where the last inequality uses again the at most $m$ entries of $\xv^{(i)}_{\diamond}$ are fractional (more precisely, uses \eqref{eq:fractional}). This is at most $k^{1/4} = (m \log n)^{1/4}$, by the {assumption on $n$}, which gives $\|\vv\|_{\infty} \le k^{1/2-\e}$ (with $\e = \frac{1}{4}$) as desired. This concludes the proof of the claim.
\end{proof}


Since each $S_t$ is independent of each other,  
aggregating over the subsets $S_t$ boosts the probability of
finding an appropriate set of coordinates, yielding the following claim, 
where the definition of $f'$ gives $ \big(1 - \frac{1}{2^{m+2}}\big)^{f'} = \big(1 - \frac{1}{2^{m+2}}\big)^{\frac{f\alpha^m}{2^{m+6}m}} = \big(\big(1 - \frac{1}{2^{m+2}}\big)^{\frac{\alpha^m}{2^{m+6}m}}\big)^f = \pi^{f}$. 

\begin{claim} \label{claim:useDiscrepancy2}
    Conditioned on $Cond$, with probability at least $1 - {\frac{1}{n^2}-\pi^f}$ there is a set $T^{(i)}$ of size $m \log n$ such that $\sum_{j \in T^{(i)}} A^{(i),j} \in \big[\zv^{(i)} - \frac{\ones}{n},~\zv^{(i)}\big]$ and $\cv^{(i)}_j - \ip{\colij}{\mu^{(i)}_\diamond} \ge - \frac{f \log n}{n}$ for all $j \in T^{(i)}$.
\end{claim}

Therefore, we use the set $T^{(i)} \subseteq S$ prescribed by this claim.
When its guarantees hold, we immediately obtain~\eqref{eq:mainDeltaItem2}.
Indeed, the binary solution
\[
\bar{\xv}^{(i)} = \lfloor \xv^{(i)}_{\diamond} \rfloor + \chi_{T^{(i)}}
\]
satisfies
\[
A^{(i)} \bar{\xv}^{(i)}
= \sum_{j \in T^{(i)}} A^{(i),j}
  + A^{(i)} \lfloor \xv^{(i)}_{\diamond} \rfloor
\le \zv^{(i)} + A^{(i)} \lfloor \xv^{(i)}_{\diamond} \rfloor
= A^{(i)} \xv^{(i)}_{\diamond} + d \ones .
\]
In addition, this solution satisfies~\eqref{eq:mainDeltaItem3}:
\begin{align*}
\Delta^{(i)}(\bar{\xv}^{(i)})
&=
\normm{\mu^{(i)}_\diamond}_1 \cdot
\normm{A^{(i)} (\xv^{(i)}_\diamond - \bar{\xv}^{(i)}) + d \ones}_{\infty}
+ \sum_{j \in T^{(i)}}
\Big( \ip{\colij}{\mu^{(i)}_\diamond} - \cv^{(i)}_j \Big) \\
&\le
\frac{m}{2\alpha} \cdot \frac{1}{n}
+ m \log n \cdot \frac{f \log n}{n}
\;\le\;
O_m\!\left(\frac{f \log^2 n}{n}\right),
\end{align*}
where we used:
(i) $\normm{\mu^{(i)}_\diamond}_1 \le m \normm{\mu^{(i)}_\diamond}_\infty
\le \frac{m}{2\alpha}$, with the last inequality following from the definition
of the event $Cond$;
(ii) the bound from Claim~\ref{claim:useDiscrepancy2};

\paragraph{Wrapping up.}
Combining the two cases
$\normm{\mu_{\diamond}^{(i)}}_1 < \frac{f \log n}{2\alpha n}$ and
$\normm{\mu_{\diamond}^{(i)}}_1 \ge \frac{f \log n}{2\alpha n}$ considered above,
we conclude that for a fixed $i$, conditioned on $Cond$, with probability at
least $1 - \frac{1}{n^2} - \pi^f$, the constructed solution
$\bar{\xv}^{(i)}$ satisfies both~\eqref{eq:mainDeltaItem2}
and~\eqref{eq:mainDeltaItem3}.

Finally, applying a union bound over all $i \in [s] \cup \{0\}$, we obtain that,
with probability at least
$1 - \frac{s+1}{n^2} - (s+1)\pi^f$, conditioned on $Cond$, we can simultaneously
construct solutions $\bar{\xv}^{(i)}$ for all blocks $i$ satisfying these
properties.
This completes the proof of Lemma~\ref{lemma:mainDeltai}.

\section{Conclusion}

This paper provides a first rigorous average-case analysis of decomposition-based methods for two-stage stochastic integer programs, with a particular focus on dual decomposition and the Branch-and-Price framework. Under a natural stochastic-input model, we show that the Branch-and-Price search tree has quasi-polynomial size with high probability, substantially improving upon the classical worst-case exponential bounds. A key ingredient in the analysis is an average-case bound on the integrality gap of the natural LP relaxation, which shows that the gap is small even when the number of scenarios grows.
An important future direction is to develop a comparable average-case analysis for Benders-type methods. While this paper focuses on Branch-and-Price and related dual decomposition schemes, Benders decomposition is another widely used approach for two-stage stochastic integer programs. Establishing average-case guarantees for Benders-type algorithms would be an important step toward a more unified theoretical understanding of why decomposition methods perform so well in practice.

\clearpage

\appendix
\noindent {\LARGE \bf Appendix}

\appendix

\section{Probabilistic Inequalities} \label{app:prob}

We will need both the additive and multiplicative Chernoff inequalities. 

\begin{lemma} \cite{vershynin2018high} \label{lemma:chernoff}
    Let $X_1,\ldots,X_k$ be independent random variables in $[0,1]$. Let $S = \sum_i X_i$ be their sum and $\mu := \E S$  its expected value. Then for any $\lambda > 0$ and $\delta \in [0,1]$ we have:
    \begin{align*}
        &\Pr\big(S \ge \mu + \lambda \sqrt{k}\big) \,\le\, e^{- 2\lambda^2}  \textrm{ , } \Pr\big(S \le \mu - \lambda \sqrt{k} \big) \,\le\, e^{- 2\lambda^2}  \textrm{, and } \Pr\big(S \le (1-\delta) \mu \big) \,\le\, e^{- \frac{\delta^2 \mu}{2}}          
    \end{align*}
\end{lemma}


\section{Proof of Claim \ref{claim:epsNet}} \label{app:epsNet}

We construct the desired $\frac{\log n}{\alpha n}$-net (w.r.t. the norm $\|\cdot\|_1$) of the set $U := \{ \mu \in \R^m_+ : \|\mu\|_{\infty} \le \frac{1}{2\alpha}, \|\mu\|_1 \ge \frac{f \log n}{2 \alpha n}\}$ as follows. 

First, consider the grid with spacing $\frac{\log n}{\alpha n m}$ of the interval $[0,\frac{1}{2\alpha}]$, plus the endpoint $\frac{1}{2\alpha}$, namely the set $G := \big((\frac{\log n}{\alpha n m} \cdot \mathbb{Z}) \cap [0,\frac{1}{2\alpha}]\big) \cup \{\frac{1}{2\alpha}\} = \{0, \frac{\log n}{\alpha n m}, \frac{2\log n}{\alpha n m}, \ldots, \frac{1}{2\alpha}\}$. Then define the $\frac{\log n}{\alpha n}$-net as $\Lambda := (G \times G \times \ldots \times G) \cap U$, namely taking the $m$-dimensional product grid and intersecting with the set $U$. 

Clearly $\Lambda$ is contained in $U$ and satisfies the desired size upper bound $|\Lambda| \le (\frac{mn}{\log n} + 1)^m$. Now we prove its approximation property, namely consider any $u \in U$ and we will exhibit some $v \in \Lambda$ such that $\|u-v\|_1 \le \frac{\log n}{\alpha n}$. Indeed, let $v$ be so that its coordinate $v_j$ is the first point in $G$ greater than or equal to $u_j$; since $u_j \le \frac{1}{2\alpha}$ and we include $\frac{1}{2\alpha}$ in the grid $G$, such point always exists. Moreover, since $v \ge u \ge 0$ we have $\|v\|_1 \ge \|u\|_1 \ge \frac{f \log n}{2\alpha n}$, and so $v$ belongs to $\Lambda$. Finally, by construction of the grid we have $u_j \le v_j \le u_j + \frac{\log n}{\alpha n m}$, and hence $\|u-v\|_1 \le \frac{\log n}{\alpha n}$. So $\Lambda$ is the desired net.


\section{Number of Fractional Coordinates of $\xv^{(i)}_\diamond$} \label{app:fractional}


We observe the following result:
\begin{lem}
    $\xv^{(i)}_\diamond$ has at most $m$ fractional entries for $i \in [s] \cup \{0\}$.
\end{lem}

\begin{proof}
Since $A^{(i)} \in \R^{m \times n}$,
for $i \in [s] \cup \{0\}$, there are at most $m$ independent rows of $A^{(i)}$. Every vertex of the LP relaxation is a solution of a linear system where some constraints of (\ref{eq:sub_LP_relax}) and some constraints of (\ref{eq:LPLast}) are active. Since (\ref{eq:LPLast}) are bound constraints, each active constraint of (\ref{eq:LPLast}) fixes some variable to be $1$ or $0$. For any such linear system defining a vertex,
after eliminating all active constraints of (\ref{eq:LPLast}) and the corresponding variables, the remaining linear system takes form of
\begin{align*}
    \begin{bmatrix}
        W^{(01)} & W^{(1)} & 0 & \cdots & 0 \\
        W^{(02)} & 0 & W^{(2)} & \cdots & 0 \\
        \cdots & 0 & 0 & \cdots & 0 \\
        W^{(0s)} & 0 & 0 & \cdots & W^{(s)} \\
    \end{bmatrix} \xv_{\text{rem}} = \bv_{\text{rem}}
\end{align*}
where $\xv_{\text{rem}},\bv_{\text{rem}}$ are the remaining variables/right-hand-sides and each row of  $W^{(i)}$ is from rows of $A^{(i)}$ and each row of $W^{(0i)}$ is from rows of $A^{(0)}$ that are active. Note that there are at most $m$ linearly independent rows from each $A^{(i)}$ or $A^{(0)}$. Therefore $\rank\left(W^{(i)}\right),\rank\left(  \begin{bmatrix}
        W^{(01)}  \\
        W^{(02)} \\
        \cdots \\
        W^{(0s)} \\
    \end{bmatrix} \right)\leq m$. Since the above system is invertible, each $W^{(i)}$ and  $W^{(0i)}$ has at most $m$ columns.
Therefore at most $m$ entries of $\xv_{\diamond}^{(i)}$ are included in $\xv_{\text{rem}}$, which is potentially set to be fractional.
\end{proof}


\section{Proof of Claim \ref{claim:good_col}} \label{app:claimGoodCol}

We first verify that indeed $M_{\mu^{(i)}}^{(i)} \subseteq \nrc_{\mu^{(i)}}^{(i)}$, for which it suffices to check that $M^{(i)}$ is contained in the unit cube $[0,1]^{1+m}$. This is clearly the case, by definition, for the last $m$ coordinates, so it suffices to verify it for the first coordinate, namely that for every $y \in [\alpha, 2\alpha]^m$, the interval $ [\ip{y}{\mu^{(i)}} - \frac{f\log n}{n}, \ip{y}{\mu^{(i)}}]$ is contained in $[0,1]$, which is indeed the case: by the lower bound on $\mu^{(i)}$ we have $\ip{y}{\mu^{(i)}} - \frac{f\log n}{n} \ge \alpha \|\mu^{(i)}\|_1 - \frac{f\log n}{n} \ge 0$, and by the upper bound on $\mu^{(i)}$ we have $\ip{y}{\mu^{(i)}} \le 2 \alpha \|\mu^{(i)}\|_{\infty} \le 1$. 

We now prove the different items in the claim.

\textbf{Item 1.} The bound on the reduced cost follows directly by the definition of the region $M^{(i)}_{\mu^{(i)}}$: if $(z,y)$ belongs to this region, then $z \ge \ip{y}{\mu^{(i)}} - \frac{f\log n}{n}$, or equivalently, $z - \ip{y}{\mu^{(i)}} \ge -\frac{f\log n}{n}$, as desired. 

\textbf{Item 2.} This also follows directly from the definition of $M^{(i)}_{\mu^{(i)}}$.

\textbf{Item 3.} 
The volume of $\nrc_{\mu^{(i)}}^{(i)}$ is at most $\min\{\|\mu^{(i)}\|_1, 1\}$, because $(z,y) \in \nrc_{\mu^{(i)}}^{(i)}$ means 
    \begin{align*}
      & z - \ip{y}{\mu^{(i)}} < 0 
      \,\implies\,  z < \norm{y}_\infty \|\mu^{(i)}\|_1 \le \|\mu^{(i)}\|_1.
    \end{align*}
and also, by definition, $z \le 1$. Combining this observation with Item 2 gives Item 3, and concludes the proof of the claim.

\label{sec:dual_decomp}

 \addcontentsline{toc}{chapter}{References}
\bibliographystyle{plain}
\bibliography{ref}

\end{document}